\documentclass[12pt]{article}

\usepackage{authblk}

\title{Existence and stability of time periodic solutions to 
nonlinear elastic wave equations \\
with viscoelastic terms}
\author[1]{Yoshiyuki Kagei}
\author[2]{Hiroshi Takeda\thanks{Corresponding author}}
\affil[1]{Department of Mathematics, 
Tokyo Institute of Technology, 
Meguro-ku, Tokyo, 152-8551, Japan}
\affil[2]{Department of Applied Mathematics, 
Faculty of Science, Fukuoka University, 
Nanakuma, Jonan-ku, Fukuoka, 814-0180, Japan}
\date{}

\usepackage{amssymb,amsmath,amsthm}
\usepackage[dvips]{graphicx}
\usepackage{bm} 

\newcommand{\R}{\mathbb R}

\newtheorem{thm}{Theorem}[section]

\newtheorem{prop}[thm]{Proposition}
\newtheorem{lem}[thm]{Lemma}
\theoremstyle{remark}

\theoremstyle{definition}

\begin{document}
\maketitle

\numberwithin{equation}{section}

\begin{abstract}
Existence and stability of time periodic solutions for nonlinear elastic wave equations with viscoelastic terms are established.
The existence of the time periodic solution is proved using the spectral decomposition of the linear principal part and the Poincar\'e map.
On the other hand, the proof of the stability of the time periodic solutions is generally problematic due to the slow time decay induced by the time periodic solutions. 
Based on the regularity estimates of the time periodic solution derived from the smoothing effect of the semigroup, 
sharp decay properties of the perturbation from the time periodic solution are proved,
which proves the stability.  
\end{abstract}

\noindent
\textbf{Keywords: }nonlinear elastic wave equation, damping terms, time periodic solution, stability \\
\noindent
\textbf{MSC2020: }Primary 35L72; Secondary 35B10, 35B65

\newpage
\section{Introduction}

In this paper, 
we investigate the system of quasi-linear hyperbolic equations with viscoelastic terms
\begin{equation} \label{eq:1.1}
	\begin{split}
		& \partial_{t}^{2} u -\mu \Delta u - (\lambda + \mu) \nabla {\rm div} u -\nu \Delta \partial_{t} u  =F(u)+g(t), \quad t \in \R, \quad x \in \R^{3}, \\
	\end{split}
\end{equation}
where, $u=u(t,x)={}^t\! (u_{1}, u_{2}, u_{3}):(0,\infty) \times \R^{3} \to \R^{3}$ is the unknown function; and the superscript ${}^t\! \cdot$ means the transpose of the matrix. 
The Lam\'e constants are supposed to satisfy  
\begin{equation} \label{eq:1.2}
\mu>0, \quad \lambda + 2 \mu >0, 
\end{equation}
and the viscosity parameter $\nu$ is a positive constant. 
We assume that the nonlinear term $F(u)$ is given by 
\[
\nabla u \nabla^{2} u,
\]
where $\nabla$ is the spatial gradient;
and $g=g(t,x):\R \times \R^{3} \to \R^{3}$ is a given time periodic external force,
that is,
$g$ has the form 
$g={}^t\! (g_{1}, g_{2}, g_{3})$ 
satisfying that there exists $T>0$ such that 
\begin{equation} \label{eq:1.3}
\begin{split}
g(t+T,x) =g(t,x)
\end{split}
\end{equation}
for $(t,x) \in \mathbb{R} \times \mathbb{R}^{3}$.
The viscoelastic body under consideration is deformed by a time periodic force. 
Our question is thus whether a time periodic deformation state of the viscoelastic body exists. 
In other words, 
we are interested in the existence and the stability of time periodic solution.
To the authors' best knowledge, 
so far there seems no result on the existence of time periodic solutions for \eqref{eq:1.1}. 

If $g=0$, several authors considered the Cauchy problem
\begin{equation} \label{eq:1.4}
\left\{
\begin{split}
& \partial_{t}^{2} u -\mu \Delta u - (\lambda + \mu) \nabla {\rm div} u -\nu \Delta \partial_{t} u  =F(u), \quad t>0, \quad x \in \R^{3}, \\
& u(0,x)=f_{0}(x), \quad \partial_{t} u(0,x)=f_{1}(x) , \quad x \in \R^{3}, 
\end{split}
\right.
\end{equation}
where $f_{j}={}^t\! (f_{j1}, f_{j2}, f_{j3})$ $(j=0,1)$ are given initial data.
For the detail see \cite{Ponce} and \cite{J-S} and the references therein.
When $\lambda + \mu=0$, based on the combination of the energy method and the $L^{p}$-$L^{q}$ type estimates,  
Ponce \cite{Ponce} proved the existence of the global solution for the small initial data
$(f_{0}, f_{1}) \in \{ \dot{H}^{s_{0}+1} \cap \dot{H}^{1} \cap \dot{W}^{1,1} \times \dot{H}^{s_{0}} \cap L^{1} \}^{3}$ with $s_{0}>\frac{7}{2}$ and obtained decay properties in the $L^{2}$-Sobolev spaces.
Jonov-Sideris \cite{J-S} also considered \eqref{eq:1.4} with $\lambda+\mu=0$ 
for $F(u)=F_{0}(u)+ \delta F_{1}(u)$, 
where $F_{0}(u)$ fulfills the Klainerman null condition (cf. \cite{K-Si}, \cite{A}, \cite{S})  
and $\delta F_{1}(u)$ represents the small perturbation.
They classified the lifespan of the solutions by the size of the ratio $\frac{\varepsilon \delta}{\nu}$ 
where $\varepsilon$ is the norms of the initial data in the suitable weighted $L^{2}$-Sobolev spaces, roughly speaking, 
$\varepsilon (1+|x|^{s_{0}})(f_{0}, f_{1}) \in \{ H^{s_{0}+1} \}^{3} \times \{ H^{s_{0}} \}^{3}$ with $s_{0} \ge 11$.
Especially, if $\frac{\varepsilon \delta}{\nu}$ is small, then the solution of \eqref{eq:1.4}  exists globally in time.
As for the situation including $\lambda + \mu \neq 0$, Kagei-Takeda \cite{K-T1} proved the existence of the global solutions for the small initial data 
$(f_{0}, f_{1}) \in \{ 
\dot{H}^{3} \cap \dot{H}^{1} \cap \dot{W}^{1,1} 
\}^{3} \times \{ \dot{H}^{1} \cap L^{1} \}^{3}$, obtained decay properties in the $L^{2}$-Sobolev spaces, 
smoothing estimates of the global solutions and asymptotic profiles as $t \to \infty$. 
Moreover they concluded the stability of $u=0$ to the Cauchy problem \eqref{eq:1.4} in the $L^{p}$-Sobolev spaces in \cite{K-T2}. 

In this paper, 
we prove the existence of a $T$-periodic solution, $u_{per}$, 
to the system \eqref{eq:1.1} in the class 
$$
\{ C([0,T];  \dot{W}^{3,p_{0}} \cap \dot{H}^{3} \cap \dot{H}^{1} ) 
\cap C^{1}([0,T]; W^{1, p_{0} }  \cap H^{1} ) \}^{3}
$$
for arbitrarily fixed $p_{0} \in [2, \infty)$
under the smallness condition on the external force $g$;
\begin{equation} \label{eq:1.5}
\begin{split}
\| g \|_{L^{1}(0,T; \dot{W}^{1,p_{0}} \cap \dot{H}^{1} \cap L^{1})} \ll 1.
\end{split}
\end{equation}
The time periodic solution satisfies the estimate 
\begin{equation} \label{eq:1.6}
\begin{split}
\| u_{per}(t) \|_{\dot{W}^{3,p_{0}} \cap \dot{H}^{3} \cap \dot{H}^{1}} 
+
\| \partial_{t} u_{per}(t) \|_{W^{1,p_{0}} \cap H^{1}} \le
C \| g \|_{L^{1}(0,T;\dot{W}^{1,p_{0}} \cap \dot{H}^{1} \cap L^{1})}. 
\end{split}
\end{equation}

Choosing $p_{0}=\frac{5}{2}$,
we conclude the stability of the time periodic solution by deriving sharp decay estimates of the perturbation $\tilde{u}= u-u_{per}$ for sufficiently small initial perturbations
 $(\tilde{u}(0), \partial_{t} \tilde{u}(0))=O(\varepsilon_{1})$ in the class 
 $\{ \dot{W}^{3,\frac{5}{2}} \cap \dot{H}^{3} \cap \dot{W}^{1,1} \}^{3} \times \{ \dot{W}^{1,\frac{5}{2}} \cap H^{1} \cap L^{1} \}^{3}$ with $\varepsilon_{1} \ll 1$
 and the external force $g$ with the assumption \eqref{eq:1.5}.
That is, we show the estimates 
\begin{align} 
			\| \nabla^{3} \tilde{u} (t) \|_{q} & 
			\le C \varepsilon_{1} (1+t)^{-\frac{3}{2}(1-\frac{1}{q})+\frac{1}{q}-\frac{3}{2}}, \quad q=\frac{5}{2}, 2, \label{eq:1.7} \\
			\| \nabla \tilde{u} (t) \|_{q} & 
			\le C \varepsilon_{1} (1+t)^{-\frac{3}{2}(1-\frac{1}{q})+\frac{1}{q}}, \quad q=\infty, 2,
			\label{eq:1.8} \\
			\| \nabla^{\alpha}  \partial_{t} \tilde{u} (t) \|_{q} & 
			\le C \varepsilon_{1} (1+t)^{-\frac{3}{2}(1-\frac{1}{q})+\frac{1}{q}-\frac{\alpha+1}{2}}, \quad q=\frac{5}{2}, 2,\, 0 \le \alpha \le 1 \label{eq:1.9}
\end{align}
for $t \ge 0$.
Here we note that the estimates \eqref{eq:1.7}-\eqref{eq:1.9} are sharp. 
Indeed,
the argument in \cite{K-T1} and \cite{K-T2} yields that
the powers in the right hand side in \eqref{eq:1.7}-\eqref{eq:1.9} are coincide with the ones of the estimates of the global solutions $u$ to \eqref{eq:1.4} 
with the small initial data with same regularity  $(f_{0}, f_{1}) \in \{ \dot{W}^{3,\frac{5}{2}} \cap \dot{H}^{3} \cap \dot{W}^{1,1} \}^{3} \times \{ \dot{W}^{1,\frac{5}{2}} \cap H^{1} \cap L^{1} \}^{3}$
which satisfy
\begin{align*} 
			C t^{-\frac{3}{2}(1-\frac{1}{q})+\frac{1}{q}-\frac{3}{2}} \le \| \nabla^{3} u (t) \|_{q} & 
			\le C t^{-\frac{3}{2}(1-\frac{1}{q})+\frac{1}{q}-\frac{3}{2}}, \quad q=\frac{5}{2}, 2,  \\
			C t^{-\frac{3}{2}(1-\frac{1}{q})+\frac{1}{q}} \le \| \nabla u (t) \|_{q} & 
			\le C t^{-\frac{3}{2}(1-\frac{1}{q})+\frac{1}{q}}, \quad q=\infty, 2,\\
			Ct^{-\frac{3}{2}(1-\frac{1}{q})+\frac{1}{q}-\frac{\alpha+1}{2}} \le \| \nabla^{\alpha}  \partial_{t} u (t) \|_{q} & 
			\le C t^{-\frac{3}{2}(1-\frac{1}{q})+\frac{1}{q}-\frac{\alpha+1}{2}}, \quad q=\frac{5}{2}, 2,\, 0 \le \alpha \le 1
\end{align*}
for $t \gg 1$.
Roughly speaking, the global solution $u$ to \eqref{eq:1.4} behaves like the diffusion wave defined by 
\begin{equation*} 
\begin{split}
G(t,x)
:=\mathcal{F}^{-1} \left[
e^{-\frac{\nu |\xi|^{2}}{2}t} \frac{\sin (|\xi| t)}{|\xi| }
\right]
\end{split}
\end{equation*}
as $t \to \infty$.
For the decay properties of the diffusion wave, we refer to \cite{H-Z1}, \cite{H-Z2}, \cite{K-S} and \cite{I1}.

To give a proof of the existence of the time periodic solution, we use 
the notion of the Poincar\'e map.
Applying Duhamel's formula (cf \cite{C-H}) to \eqref{eq:1.1} with the semigroup 
$\{ e^{t \mathcal{A}} \}_{t \ge 0}$ generated by
\begin{equation*}
\begin{split}
\mathcal{A} :=
\begin{pmatrix}
0 & I_{3} \\
\mu \Delta I_{3} +(\lambda +\mu) \nabla {\rm div} & \nu \Delta I_{3}  
\end{pmatrix},
\end{split}
\end{equation*}
where $I_{3}:= \text{diag}(1,1,1)$,
we have the expression of the solution as follows:
\begin{equation} \label{eq:1.10}
\begin{split}
{\bf u}(t) = e^{t \mathcal{A} } {\bf u}(0)+ \int_{0}^{t} e^{(t-s) \mathcal{A} } 
({\bf F}({\bf u})(s) + {\bf G}(s) )ds,
\end{split}
\end{equation}
where
\begin{equation*}
\begin{split}
{\bf u} :=
\begin{pmatrix}
u  \\
\partial_{t} u
\end{pmatrix},
\qquad 
{\bf F}({\bf u}) := 
\begin{pmatrix}
0 \\
{F}(u)
\end{pmatrix}, \qquad 
{\bf G} := 
\begin{pmatrix}
0 \\
g(t)
\end{pmatrix}.
\end{split}
\end{equation*}
We look for a solution of \eqref{eq:1.10} satisfying
\begin{equation}\label{*}
{\bf u}(T)={\bf u}(0).
\end{equation}
By the condition \eqref{*}, 
we have
\begin{equation} \label{eq:1.11} 
\begin{split}
{\bf u}(0)=   (I_{6}-e^{T \mathcal{A} })^{-1} \int_{0}^{T} e^{(T-s) \mathcal{A} } 
\left( {\bf F} ({\bf u})(s) 
+ 
{\bf G}(s) 
\right)
ds.
\end{split}
\end{equation}
Now we denote 
\begin{equation*}
\begin{split}
\hat{\mathcal{A}}_{\xi} :=
\begin{pmatrix}
0 & I_{3} \\
-\mu |\xi|^{2} I_{3} -(\lambda +\mu) \xi \otimes \xi & -\nu |\xi|^{2} I_{3}  
\end{pmatrix}
\in M(\mathbb{R};6).
\end{split}
\end{equation*}
Then by the formal observation 
\begin{equation*} 
\begin{split}
 I_{6}-e^{T \hat{\mathcal{A}}_{\xi} } \sim -T \hat{\mathcal{A}}_{\xi} + O(|\xi|^{4})
\end{split}
\end{equation*}
for $|\xi| \ll 1$ and denoting $R_{1}:= \frac{\xi}{|\xi|} \otimes \frac{\xi}{|\xi|} \in M(3;\mathbb{R})$, $R_{2}:=I_{3}- R_{1} \in M(3;\mathbb{R})$ and 
\begin{equation*} 
\begin{split}
\tilde{R}:= \frac{1}{\lambda+2 \mu} R_{1}  +
\frac{1}{\mu} R_{2},
\end{split}
\end{equation*}
we see that 
\begin{equation*} 
\begin{split}
(I_{6}-e^{T \hat{\mathcal{A}}_{\xi} })^{-1} \sim -T^{-1} \hat{\mathcal{A}}_{\xi}^{-1}
=
-\frac{1}{T}
\begin{pmatrix}
\nu \tilde{R} & -\frac{1}{|\xi|^{2}} \tilde{R} \\
I_{3}  & 0
\end{pmatrix}
\sim 
-\frac{1}{T}
\begin{pmatrix}
\nu I_{3} & -\widehat{(-\Delta)^{-1}} I_{3} \\
I_{3}  & 0
\end{pmatrix} 
\end{split}
\end{equation*}
as $|\xi| \to 0$,
and 
\begin{equation*} 
\begin{split}
 (I_{6}-e^{T \hat{\mathcal{A}}_{\xi} })^{-1} \sim I_{6}
\end{split}
\end{equation*}
as $|\xi| \to \infty$, which implies that the high frequency part of ${\bf u}(0)$ is tractable. 
The argument here will be rigorously justified in Section \ref{sec:2} by the spectral decomposition of 
$\hat{\mathcal{A}}_{\xi}$ in the Fourier space.
Therefore, we formally obtain the equation of the low frequency part ${\bf u}_{L}$ of \eqref{eq:1.10} with \eqref{eq:1.11} by
\begin{equation} \label{eq:1.12}
\begin{split}
{\bf u}_{L}(t) \sim \int_{0}^{T} \frac{e^{(t+T-s) \mathcal{A} } }{-\Delta} ({\bf F}({\bf u})(s) + {\bf G}(s))_{L} ds+ \int_{0}^{t} e^{(t-s) \mathcal{A} } ({\bf F}({\bf u})(s) + {\bf G}(s))_{L} ds,
\end{split}
\end{equation}
and the one of the high frequency part ${\bf u}_{H}$ of \eqref{eq:1.10} as follows:
\begin{equation} \label{eq:1.13}
\begin{split}
{\bf u}_{H}(t) \sim \int_{0}^{T} e^{(t+T-s) \mathcal{A} } ({\bf F}({\bf u})(s) + {\bf G}(s))_{H} ds
+\int_{0}^{t} e^{(t-s) \mathcal{A} } ({\bf F}({\bf u})(s) + {\bf G}(s))_{H} ds,
\end{split}
\end{equation}
where ${\bf u}={\bf u}_{L}+{\bf u}_{H}$; 
and $({\bf F}({\bf u})(s) + {\bf G}(s))_L$ and $({\bf F}({\bf u})(s) + {\bf G}(s))_H$ denote the low and high frequency parts of $({\bf F}({\bf u})(s) + {\bf G}(s))$, 
respectively.

It is not difficult to see that the leading factor in the sense of regularity of the periodic solution is given by the first component in \eqref{eq:1.12} and \eqref{eq:1.13}.
Therefore we investigate the operator $Q(t) \ast$ formally defined by 
\begin{equation*} 
\begin{split}
Q(t-s) \ast \sim 
\begin{cases}
& \dfrac{
e^{(t+T-s) \mathcal{A} }
}{-\Delta}, \quad |\xi| \ll 1, \\
& e^{(t+T-s) \mathcal{A}} , \quad |\xi| \gg 1.
\end{cases}
\end{split}
\end{equation*}
Precise definition of $Q(t) \ast$ is given in \eqref{eq:2.5} below.
For $k=2,3$, 
the estimates of $Q(t) \ast$ in $\dot{W}^{k,p}$ are proved by the $L^{p}$-$L^{p}$ boundness of the Riesz transform (Lemma \ref{Lem:3.2}) and estimates of the semigroup $e^{t \mathcal{A}}$ (Lemma \ref{Lem:3.1}).
On the other hand, we apply the Hardy-Littlewood-Sobolev inequality (Lemma \ref{Lem:3.3}) to show the low frequency estimates of $Q(t) \ast$ in $\dot{W}^{1,p}$ to ensure the integrability near $|\xi|=0$.
Based on the estimates of $Q(t) \ast$ in $\dot{W}^{k,p}$ for $k=1,2,3$, 
we conclude the existence of the time periodic solution with the application of the contraction mapping principle.

For the proof of the stability of the time periodic solution, 
we assume $p_{0}=\frac{5}{2}$.
Indeed, the perturbation $\tilde{u}$ is expected to satisfy same decay properties in 
$\dot{W}^{3, \frac{5}{2}}$ and $\dot{W}^{1, \infty}$ by \eqref{eq:1.7} with $q=\frac{5}{2}$
and \eqref{eq:1.8} with $q=\infty$, 
which is useful to obtain sharp decay properties from the high frequency part of $\tilde{u}$. 
For the detail, see subsection 6.2 below. 

In the case $p_{0}=\frac{5}{2}$,
we also note that the regularity estimates of the time periodic solution $u_{per}$ are stated as follows: 
\begin{equation} \label{eq:1.14}
\begin{split}
\| \nabla u_{per}(t)  \|_{p} & \le C, \quad \frac{3}{2} < p \le \infty, \\
\| \nabla^{2} u_{per}(t)  \|_{p} & \le C, \quad 1 < p \le 15, \\
\| \nabla^{3} u_{per}(t)  \|_{p} & \le C, \quad 1 < p \le \frac{5}{2}
\end{split}
\end{equation}
for $t \in [0,T]$. 
The point in \eqref{eq:1.14} is estimates in the case $p<2$ which are useful in estimating the linearization of $F(u)$ at $u_{per}$ .
The proof of \eqref{eq:1.14} is based on the estimates of $Q(t) \ast$ derived in Section 4, the smoothing effect of the semigroup $\{ e^{t \mathcal{A}} \}_{t \ge 0}$ by \cite{Shibata} and \cite{K-S} for the low frequency part \eqref{eq:1.13} and by \cite{K-T1} and \cite{K-T2} for the high frequency part \eqref{eq:1.14}.
Especially, the observation for the high frequency part of the semigroup in \cite{K-T2}
is applied to show Proposition \ref{prop:4.4} which plays an important role to have the estimates \eqref{eq:1.14} for $p<2$.  

On the other hand,
recalling that $u=\tilde{u}+u_{per}$,
we see that
\begin{equation*}
\begin{split}
F(u)-F(u_{per})= F(\tilde{u})+\nabla \tilde{u} \nabla^{2} u_{per} 
+ \nabla u_{per} \nabla^{2} \tilde{u}.
\end{split}
\end{equation*}
To obtain the estimates \eqref{eq:1.7}-\eqref{eq:1.9},
we need to derive the sharp decay properties from $F(u)-F(u_{per})$.
Such kind of situation also appears in the context of the fluid dynamical equations and the authors overcome the difficulty by the use of 
smoothing effect and space-time integral estimates of the heat semigroup 
(cf. \cite{K-N}, \cite{Taniuchi}, \cite{Yamazaki}) and 
the Hardy inequality (cf. \cite{K-K}, \cite{K-T}, \cite{Tsuda}).
In contrast to these cases, for the proof of stability, 
we apply the regularity estimates \eqref{eq:1.14} to obtain sharp decay properties for 
$\nabla \tilde{u} \nabla^{2} u_{per} + \nabla u_{per} \nabla^{2} \tilde{u}$.
As a result, we regard them as a perturbation of the semigroup $e^{t \mathcal{A}}$
and we prove the global existence of the perturbation $\tilde{u}$ and sharp decay properties of $\tilde{u}$.
For example, the estimates \eqref{eq:1.14} enables us to have 
\begin{equation*}
\begin{split}
& \| \nabla \tilde{u}(t) \nabla^{2} u_{per}(t) \|_{1} \le 
\left\| \nabla \tilde{u}(t) \right\|_{10} \left\| \nabla^{2} u_{per}(t) \right\|_{\frac{10}{9}},\\
& \| \nabla u_{per}(t) \nabla^{2} \tilde{u}(t) \|_{1} \le \| \nabla u_{per}(t) \|_{\frac{5}{3}} 
\| \nabla^{2} \tilde{u}(t) \|_{\frac{5}{2}}
\end{split}
\end{equation*}
by the H\"{o}lder inequality. 
Proving stability from this perspective is our mathematical contribution.

This paper is organized as follows.
In section 2, we state the main results precisely. 
To do so, we begin with the summary of the notation and formulation of the problem.
Section 3 is devoted to the preliminaries where we lists the facts used throughout this paper, 
the estimates for the evolution operators for the linearized problem and the summary of the basic inequalities of the real analysis .
In section 4, we obtain the $L^{p}$-$L^{q}$ type estimates of the evolution operator $Q(t) \ast$, which plays an essential role to show Theorems \ref{thm:2.1} and \ref{thm:2.3}.  
We prove the main results in sections 5-6. 
In section 5, we derive the existence theorem and regularity estimates of the time periodic solutions corresponding the regularity of the external force $g$.
Section 6 establishes the stability of the time periodic solution under the assumption 
\eqref{eq:1.5} with $p_{0}=\frac{5}{2}$,
by proving the estimates 
\eqref{eq:1.7}-\eqref{eq:1.9}.

\section{Main results} \label{sec:2}
In this section, 
we state the main results of this paper, rigorously.
For this aim, 
we firstly set up the notation which will be used throughout the paper and 
formulate the periodic solutions to \eqref{eq:1.1} by an integral equation.

\subsection{Notation}
In this subsection we summarize standard notations and notions.
We denote by $L^{p}=L^{p}(\R^{3})$ the usual Lebesgue space equipped with the norm $\| \cdot \|_{p}$ for $1 \le p \le \infty$. 
The symbol $W^{k,p}(\R^{3})$ denotes the $L^{p}$-Sobolev spaces equipped with the norm $\| \cdot \|_{W^{k,p}(\R^{3})} := \| \cdot \|_{p}+\|  \nabla_{x}^{k} \cdot \|_{p}$.
We denote the homogeneous Sobolev space $\dot{W}^{k,p}(\R^{3})$ for $k \in \mathbb{N}$ and $1\le p \le \infty$ by
\begin{equation*}
\dot{W}^{k,p}(\mathbb{R}^{3}) := 
\overline{
C^{\infty}_{0}(\R^{3})
}^{\| \cdot \|_{\dot{W}^{k,p}(\mathbb{R}^{3})}},
\end{equation*}
where $\overline{
C^{\infty}_{0}(\R^{3})
}^{\| \cdot \|_{\dot{W}^{k,p}(\mathbb{R}^{3})}}$ 
stands for the completion of $C^{\infty}_{0}(\R^{3})$ by the norm $\| \cdot \|_{\dot{W}^{k,p}(\mathbb{R}^{3})}$ and 
$\| \cdot \|_{\dot{W}^{k,p}(\mathbb{R}^{3})} := 
 \|\nabla_{x}^{k} \cdot \|_{p}$.
When $p=2$, we define $H^{k}(\R^{3})=W^{k,2}(\R^{3})$ and $\dot{H}^k(\mathbb{R}^{3}) := \dot{W}^{k,2}(\mathbb{R}^{3})$ for $k \ge 0$.
We often abbreviate the domain $\R^{3}$ in the notation of the function spaces and norms, 
for short.
For an interval $I \subset \R$ and a Banach space $X$, 
we denote by $C^{k}(I;X)$ the space of $X$-valued $k$-times continuously differentiable functions on $I$.
We also define $C^{k}_{per}(\R;X)$ and $L^{p}_{per}(\R;X)$ by the space of $X$-valued $T$-periodic $k$-times continuously differentiable functions on $\mathbb{R}$ and 
the Bochner space of $X$-valued $T$-periodic functions on $\mathbb{R}$ for $1 \le p \le \infty$,
respectively.

The notation $\hat{f}$ means that the Fourier transform of a function $f$ defined by 
\begin{align*}
\hat{f}(\xi) := (2 \pi)^{-\frac{3}{2}}
\int_{\R^{3}} e^{-i x \cdot \xi} f(x) dx.
\end{align*}
Also, we write $\mathcal{F}^{-1}[f]$ or $\check{f}$ the inverse
Fourier transform of $f$.

\subsection{Semigroup formulation of the solution of linear problem}
In this subsection, we mention the solution formulas of the following linearized system of \eqref{eq:1.1}:
\begin{equation} \label{eq:2.1}
	\left\{
	\begin{split}
		& \partial_{t}^{2} u -\mu \Delta u - (\lambda + \mu) \nabla {\rm div} u -\nu \Delta \partial_{t} u  =0, \quad t>0, \quad x \in \R^{3}, \\
		& u(0,x)=f_{0}(x), \quad \partial_{t} u(0,x)=f_{1}(x) , \quad x \in \R^{3},
	\end{split}
	\right.
\end{equation}
introducing the notion of the semigroup generated by the linear principal parts.
We also discuss the spectral decomposition of the semigroup.

First, we define the matrices 
\begin{equation*}
\begin{split}
\mathcal{A} :=
\begin{pmatrix}
0 & I_{3} \\
\mu \Delta I_{3} +(\lambda +\mu) \nabla {\rm div} & \nu \Delta I_{3}  
\end{pmatrix}
\end{split}
\end{equation*}
and
\begin{equation*}
\begin{split}
\hat{\mathcal{A}}_{\xi} :=
\begin{pmatrix}
0 & I_{3} \\
-\mu |\xi|^{2} I_{3} -(\lambda +\mu) \xi \otimes \xi & -\nu |\xi|^{2} I_{3}  
\end{pmatrix}
\in M(\mathbb{R};6),
\end{split}
\end{equation*}
where $M(\mathbb{R};n)$ denotes the set of square matrices of size $n$ over $\mathbb{R}$
and $I_{n} \in M(\mathbb{R};n)$ stands for the identity matrix for $n \in \mathbb{N}$. 
Denoting that $v= \partial_{t} u$ and 
$${\bf u} :=
\begin{pmatrix}
u  \\
v
\end{pmatrix},$$
we have 
\begin{equation*}
\begin{split}
\frac{d}{dt} {\bf u} = \mathcal{A} {\bf u}.
\end{split}
\end{equation*}
Observing that 
\begin{equation*}
\begin{split}
\det (\sigma I_{6} -\hat{\mathcal{A}}_{\xi} ) & = (\sigma^{2} + \nu |\xi|^{2} \sigma+ (\lambda + 2 \mu ) |\xi|^{2})
(\sigma^{2} + \nu |\xi|^{2} \sigma + \mu  |\xi|^{2})^{2} \\
& =(\sigma -\sigma_{1,+})(\sigma -\sigma_{1,-})(\sigma -\sigma_{2,+})^{2}(\sigma -\sigma_{2,-})^{2},
\end{split}
\end{equation*}
we obtain
\begin{equation*}
\begin{split}
\sigma_{1, \pm} := \frac{- \nu |\xi|^{2} \pm \sqrt{ \nu^{2} |\xi|^{4} - 4 (\lambda +2 \mu) |\xi|^{2}} }{2}
\end{split}
\end{equation*}
and 
\begin{equation*}
\begin{split}
\sigma_{2, \pm} := \frac{- \nu |\xi|^{2} \pm \sqrt{ \nu^{2} |\xi|^{4} - 4 \mu |\xi|^{2}} }{2}.
\end{split}
\end{equation*}
%
%
By the characteristic roots $\sigma_{1, \pm}$ and $\sigma_{2, \pm}$, 
we have the expression of the solution of \eqref{eq:2.1} in the Fourier space as follows;
\begin{equation*}
\begin{split}
\hat{u}_{lin}(t) = \hat{K}_{0}(t,\xi) \hat{f}_{0} 
+\hat{K}_{1}(t,\xi) \hat{f}_{1},
\end{split}
\end{equation*}
where
\begin{equation*}
\begin{split}
\hat{K}_{k}(t,\xi) = \hat{K}_{k}^{(1)}(t,\xi)  +\hat{K}_{k}^{(2)}(t,\xi) 
\end{split}
\end{equation*}
for $k=0,1$ and 
\begin{equation*}
\begin{split}
\hat{K}_{0}^{(j)}(t,\xi):= 
\frac{-\sigma_{j,-} e^{\sigma_{j,+} t}+\sigma_{j,+} e^{\sigma_{j,-} t} }
{\sigma_{j,+}-\sigma_{j,-}} R_{j}, \quad 
	\hat{K}_{1}^{(j)}(t,\xi):= 
	\frac{
		e^{\sigma_{j,+} t}-e^{\sigma_{j,-}t}
	}{\sigma_{j,+}-\sigma_{j,-}} R_{j}
\end{split}
\end{equation*}
for $j=1,2$, where $R_{1}:= \frac{\xi}{|\xi|} \otimes \frac{\xi}{|\xi|} \in M(3;\mathbb{R})$, $R_{2}:=I_{3}- R_{1} \in M(3;\mathbb{R})$.
We introduce the notation
\begin{equation*}
\begin{split}
\mathcal{R}_{j} :=
\begin{pmatrix}
R_{j} & 0 \\
0 & R_{j}
\end{pmatrix}
\in M(\mathbb{R};6)
\end{split}
\end{equation*}
and
\begin{equation*}
\begin{split}
\mathcal{P}_{j,+} &:= 
\frac{1}{\sigma_{j,+}-\sigma_{j,-}}
\begin{pmatrix}
-\sigma_{j,-} I_{3} & I_{3}  \\
-\sigma_{j,+} \sigma_{j,-} I_{3} & \sigma_{j,+} I_{3}
\end{pmatrix}
\mathcal{R}_{j}, \\
\mathcal{P}_{j,-} & := 
\frac{1}{\sigma_{j,+}-\sigma_{j,-}}
\begin{pmatrix}
\sigma_{j,+} I_{3}  & -I_{3}  \\
\sigma_{j,+} \sigma_{j,-} I_{3} & -\sigma_{j,-} I_{3}
\end{pmatrix}
\mathcal{R}_{j}
\end{split}
\end{equation*}
for $j=1$, $2$.
Then it is easy to see that 
\begin{equation*}
\begin{split}
& \hat{\mathcal{A}}_{\xi} = \sigma_{1,+} \mathcal{P}_{1,+} +\sigma_{1,-} \mathcal{P}_{1,-} + \sigma_{2,+} \mathcal{P}_{2,+} + \sigma_{2,-} \mathcal{P}_{2,-}, \\
& I_{6}= \mathcal{P}_{1,+} +\mathcal{P}_{1,-} + \mathcal{P}_{2,+} + \mathcal{P}_{2,-},\\
& \mathcal{P}_{1,\pm}^{2}=\mathcal{P}_{1,\pm}, 
\quad \mathcal{P}_{2,\pm}^{2} =\mathcal{P}_{2,\pm}, \\
& \mathcal{P}_{1,\pm} \mathcal{P}_{2,\pm}  =\mathcal{P}_{2,\pm} \mathcal{P}_{1,\pm} =  0.
\end{split}
\end{equation*}
Therefore we arrive at the identities
\begin{equation} \label{eq:2.2}
\begin{split}
\hat{{\bf u}} & =
e^{t \hat{\mathcal{A}}_{\xi}} 
\begin{pmatrix}
\hat{f}_{0} \\
\hat{f}_{1}
\end{pmatrix}
,\\
e^{t \hat{\mathcal{A}}_{\xi}} 
& =\begin{pmatrix}
\hat{K}_{0}(t,\xi) 
 & \hat{K}_{1}(t,\xi)    \\
\partial_{t} \hat{K}_{0}(t,\xi) 
 & \partial_{t} \hat{K}_{1}(t,\xi) 
\end{pmatrix} \\
& = e^{\sigma_{1,+} t} \mathcal{P}_{1,+} + e^{\sigma_{1,-} t} \mathcal{P}_{1,-} + e^{\sigma_{2,+} t} \mathcal{P}_{2,+} + e^{\sigma_{2,-} t} \mathcal{P}_{2,-}
\end{split}
\end{equation}
and 
\begin{equation} \label{eq:2.3}
\begin{split}
(I_{6}-e^{t\hat{\mathcal{A}}_{\xi} })^{-1} 
= \sum_{j=1}^{2} \left\{  (1-e^{\sigma_{j,+} t})^{-1} 
\mathcal{P}_{j,+} + (1-e^{\sigma_{j,-} t})^{-1} \mathcal{P}_{j,-} 
\right\}.
\end{split}
\end{equation}
\subsection{Representation formula of the periodic solution}
This subsection is devoted to the reformulation of the system \eqref{eq:1.1} into the integral equation using the notion of the Poincar\'e map.
At first, we denote the system \eqref{eq:1.1} by  
\begin{equation} \label{**}
\begin{split}
\frac{d}{dt} {\bf u} = \mathcal{A} {\bf u}+ {\bf F}+ {\bf G},
\end{split}
\end{equation}
where 
\begin{equation*}
\begin{split}
{\bf u} :=
\begin{pmatrix}
u  \\
v
\end{pmatrix},
\qquad 
{\bf F}({\bf u}) := 
\begin{pmatrix}
0 \\
{F}(u)
\end{pmatrix}, \qquad 
{\bf G} := 
\begin{pmatrix}
0 \\
g(t)
\end{pmatrix}.
\end{split}
\end{equation*}
We look for a solution of \eqref{**} satisfying 
\begin{equation} \label{***}
u(T)=u(0).
\end{equation}
By applying the Duhamel principle and \eqref{***}, 
we have the following integral equation;
\begin{equation} \label{eq:2.4}
\begin{split}
{\bf u}(t) = e^{t \mathcal{A} } {\bf u}(0)+ \int_{0}^{t} e^{(t-s) \mathcal{A} } 
({\bf F}({\bf u})(s) + {\bf G}(s)) ds,
\end{split}
\end{equation}
where ${\bf u}(0)$ is defined by 
\begin{equation*} 
\begin{split}
{\bf u}(0)=  \left[ (I_{6}-e^{T \mathcal{A} })^{-1} \int_{0}^{T} e^{(T-s) \mathcal{A} } 
\left( {\bf F} ({\bf u})(s) 
+ 
{\bf G}(s) 
\right)
ds \right].
\end{split}
\end{equation*}

To solve \eqref{eq:2.4},
we decompose the operation of the semigroup in 
$e^{t \mathcal{A} } {\bf u}(0)$ as follows:
\begin{equation*}
\begin{split}
e^{t \hat{\mathcal{A}}_{\xi} } (I_{6}-e^{T\hat{\mathcal{A}}_{\xi} })^{-1} e^{(T-s) \hat{\mathcal{A}}_{\xi}  }  
= 
\sum_{j=1}^{2}
\left(
\dfrac{
e^{\sigma_{j,+} (t+T-s)}
}{1-e^{\sigma_{j,+} T}
}\mathcal{P}_{j,+} + 
\dfrac{e^{\sigma_{j,-} (t+T-s)} 
}
{1-e^{\sigma_{j,-} T} } \mathcal{P}_{j,-} 
\right) 
\end{split}
\end{equation*}
by \eqref{eq:2.2} and \eqref{eq:2.3}.
Denoting 
\begin{equation*}
\begin{split}
\hat{Q}^{(j)}(t, \xi) & := 
\frac{1}{\sigma_{j,+}-\sigma_{j,-}}
\left(
\frac{e^{\sigma_{j,+} (t+T)} }{ 1-e^{\sigma_{j,+} T} } -\frac{e^{\sigma_{j,-} (t+T)} }{1-e^{\sigma_{j,-} T} }
\right) R_{j}, 
\end{split}
\end{equation*}
for $j=1$, $2$ and 
\begin{equation*}
\begin{split}
Q(t,x) := Q^{(1)}(t,x)+Q^{(2)}(t,x),
\end{split}
\end{equation*}
we arrive at the following integral equation:
%
\begin{equation} \label{eq:2.5}
\begin{split}
u_{per}(t) & = \int_{0}^{T} Q(t-s) \ast (F(u_{per})(s) +g(s)) ds \\
& \quad + \int_{0}^{t} K_{1}(t-s) \ast (F(u_{per})(s) +g(s)) ds
\end{split}
\end{equation}
by \eqref{eq:2.4}.
\subsection{Statements of main results}
Our first result is the existence of a time periodic solution.  
\begin{thm} \label{thm:2.1}
Let $2 \le p_{0} < \infty$, $g \in Y_{0}:= \{  L^{1}_{per}(\R; \dot{W}^{1,p_{0}} \cap \dot{H}^{1} \cap L^{1} ) \}^{3}$.
Then there exists a constant $\varepsilon_{0}>0$ such that if
$$\| g \|_{L^{1}(0,T;\dot{W}^{1,p_{0}} \cap \dot{H}^{1} \cap L^{1} )} \le \varepsilon_{0},$$
there exists a unique periodic solution of \eqref{eq:1.1} with \eqref{eq:1.2}, $u_{per}$, 
in the sense of the integral equation \eqref{eq:2.5} in the class  
$$
\{ C_{per}(\R; \dot{W}^{3,p_{0}}  \cap \dot{H}^{3} \cap \dot{H}^{1} ) 
\cap C_{per}^{1}(\R; W^{1,p_{0}}  \cap H^{1} ) \}^{3}
$$
satisfying 
\begin{equation} \label{eq:2.6}
\begin{split}
\sum_{q= p_{0},2} (\| \nabla^{3} u_{per}(t) \|_{q} +\| \nabla \partial_{t} u_{per}(t) \|_{q} 
+\| \partial_{t} u_{per}(t) \|_{q} )+\| \nabla u_{per}(t) \|_{2} \le \varepsilon_{0},
\end{split}
\end{equation}
for $t \in [0,T]$.
\end{thm}
The proof is based on the $L^{p}$-$L^{q}$ type estimates of the evolution operators 
$K_{1}(t) \ast$ and $Q(t) \ast$ (cf. Lemma \ref{Lem:3.1}, Propositions \ref{prop:4.2} and \ref{prop:4.4}). 

We next mention the stability of the time periodic solution $u_{per}$, derived in Theorem \ref{thm:2.1}. 
We rewrite the equations into the ones for the perturbation  
$\tilde{u}:= u-u_{per}$.
Thus,
we substitute $u=\tilde{u} + u_{per}$ into the equation \eqref{eq:1.1} and obtain 
\begin{equation} \label{eq:2.7}
\begin{split}
\partial_{t}^{2} \tilde{u} -\mu \Delta \tilde{u} - (\lambda + \mu) \nabla {\rm div} \tilde{u} -\nu \Delta \partial_{t} \tilde{u}=G(\tilde{u}),
\end{split}
\end{equation}
where 
$G(\tilde{u}):= F(\tilde{u}) + \nabla \tilde{u} \nabla^{2} u_{per} +\nabla u_{per} \nabla^{2} \tilde{u}$.
Under the initial condition 
\begin{equation} \label{eq:2.8}
\begin{split}
(\tilde{u}(0), \partial_{t} \tilde{u}(0)) = (f_{0}, f_{1}),
\end{split}
\end{equation}
we show the decay estimates of the solutions,
$\tilde{u}$, 
to the Cauchy problem \eqref{eq:2.7}-\eqref{eq:2.8}.  
Taking $p_{0}=\frac{5}{2}$ in Theorem \ref{thm:2.1},
we conclude the stability of the time periodic solution under the small perturbation, which is stated as follows.
\begin{thm} \label{thm:2.2}
Let $(f_{0}, f_{1}) \in Y_{1} := \{ \dot{W}^{3,\frac{5}{2}} \cap \dot{H}^{3} \cap \dot{W}^{1,1} \}^{3} \times \{ \dot{W}^{1,\frac{5}{2}} \cap H^{1} \cap L^{1} \}^{3}$.
Then there exist a constant $\varepsilon_{1}>0$ such that if
$$\| g \|_{L^{1}(0,T;\dot{W}^{1,\frac{5}{2}} \cap \dot{H}^{1} \cap L^{1}  )} + \| f_{0}, f_{1} \|_{Y_{1}} < \varepsilon_{1},$$
the global solution $\tilde{u}$ to the Cauchy problem \eqref{eq:2.7}-\eqref{eq:2.8} with \eqref{eq:1.2}
exists uniquely in the class 
$$
\{ C([0,\infty);\dot{W}^{3,\frac{5}{2}} \cap \dot{H}^{3} \cap \dot{H}^{1}) 
\cap C^{1}([0,\infty);W^{1, \frac{5}{2}} \cap H^{1}) \}^{3}
$$
and $\tilde{u}$ satisfies the following estimates:
\begin{align} 
			\| \nabla^{3} \tilde{u} (t) \|_{q} & 
			\le \varepsilon_{1} (1+t)^{-\frac{3}{2}(1-\frac{1}{q})+\frac{1}{q}-\frac{3}{2}}, \quad q=\frac{5}{2},2, \label{eq:2.9} \\
			\| \nabla \tilde{u} (t) \|_{q} & 
			\le \varepsilon_{1} (1+t)^{-\frac{3}{2}(1-\frac{1}{q})+\frac{1}{q}}, 
			\quad q=\infty,2,
			\label{eq:2.10} \\
			\| \nabla^{\alpha}  \partial_{t} \tilde{u} (t) \|_{q} & 
			\le \varepsilon_{1} (1+t)^{-\frac{3}{2}(1-\frac{1}{q})+\frac{1}{q}-\frac{\alpha+1}{2}}, \quad q=\frac{5}{2},2,\ 0 \le \alpha \le 1 \label{eq:2.11}
\end{align}
for $t \ge 0$.
\end{thm}
The proof of Theorem \ref{thm:2.2} relies heavily on regularity estimates of the time periodic solution to \eqref{eq:1.1}.
Depending on the value of $p_{0}$, 
the regularity of the time periodic solution is classified in the following theorem.
\begin{thm} \label{thm:2.3}
Let $2 \le p_{0} < \infty$ and $u_{per}$ be the solution constructed in Theorem \ref{thm:2.1}. \\
{\rm (i)} If $p_{0} \neq 3$,
then $u_{per}$ belongs to the class 
$$
\left\{ 
C_{per}\left( \R; \left(\bigcup_{\frac{3}{2} < p \le \infty} \dot{W}^{1,p} \right) \cup \left( \bigcup_{1 < p \le p_{0 \ast}} \dot{W}^{2,p} \right)
\cup \left( \bigcup_{1 < p \le p_{0}} \dot{W}^{3,p} \right) \right)
\right\}^{3}
$$ 
and satisfies the estimates 
\begin{align}
\| \nabla u_{per}(t)  \|_{p} & \le C \varepsilon_{0}, \quad \frac{3}{2} < p \le \infty, \label{eq:2.12} \\
\| \nabla^{2} u_{per}(t)  \|_{p} & \le C \varepsilon_{0}, \quad 1 < p \le p_{0 \ast}, \label{eq:2.13} \\
\| \nabla^{3} u_{per}(t)  \|_{p} & \le C \varepsilon_{0}, \quad 1 < p \le p_{0}  \label{eq:2.14}
\end{align}
for $t \in [0,T]$, where $p_{0 \ast} := 
\begin{cases} 
& \dfrac{3 p_{0}}{3-p_{0}}, \qquad 2 \le p_{0} < 3, \\ 
& \infty,\qquad p_{0}>3. 
\end{cases}
$ \ \\
{\rm (ii)} If $p_{0} = 3$,
then $u_{per}$ belongs to the class 
$$
\left\{
C_{per}\left( \R; \left(\bigcup_{\frac{3}{2} < p \le \infty} \dot{W}^{1,p} \right) \cup \left( \bigcup_{1 < p< \infty} \dot{W}^{2,p} \right)
\cup \left( \bigcup_{1 < p \le p_{0}} \dot{W}^{3,p} \right) \right)
\right\}^{3}
$$ 
and satisfies the estimates \eqref{eq:2.12}, \eqref{eq:2.14} and 
\begin{align}
\| \nabla^{2} u_{per}(t)  \|_{p} & \le C \varepsilon_{0}, \quad 1 < p < \infty, \label{eq:2.15}
\end{align}
for $t \in [0,T]$.
\end{thm}
Especially,
the key observation is the estimates of $u_{per}$ in the Sobolev spaces $\dot{W}^{k,p}$ for $k=1,2,3$ and $p<2$.
They are useful to show the stability of the time periodic solution, when $p_{0}=\frac{5}{2}$.
%
%
%
%
%
\section{Preliminaries}
%
%
This section is devoted to collect the fundamental facts used throughout the paper.
We begin with the decay properties of the fundamental solutions of \eqref{eq:2.1}.
To do so,
we define the smooth cut-off functions $\chi_{l}= \chi_{l}(\xi) \in C^{\infty}(\R^{3})$ 
$(l=L,M,H)$ in the Fourier space as follows:
\begin{equation*}
\begin{split}
\chi_{L}:= 
\begin{cases}
& 1 \quad (|\xi|\le \frac{c_{0}}{2}), \\
& 0 \quad (|\xi|\ge c_{0}),
\end{cases}
\end{split}
\end{equation*}
\begin{equation*}
\begin{split}
\chi_{H}:= 
\begin{cases}
& 0 \quad (|\xi|\le c_{1}), \\
& 1 \quad (|\xi|\ge 2c_{1})
\end{cases}
\end{split}
\end{equation*}
and 
\begin{equation*}
\chi_{M}=1-\chi_{L}-\chi_{H}.
\end{equation*}
Here $c_{0}$ and $c_{1}$ $(0<c_{0}<c_{1}<\infty)$ are some constants to be determined later.
Then the localized fundamental solutions are defined by 
\begin{equation*} 
\begin{split}
\hat{K}_{k l}(t, \xi) & := \hat{K}_{k}(t, \xi) \chi_{l},
\end{split}
\end{equation*}
and
\begin{equation*} 
\begin{split}
\hat{K}_{k l}^{(j)}(t, \xi) & := \hat{K}_{k}^{(j)}(t, \xi) \chi_{l}
\end{split}
\end{equation*}
for $j=1,2$, $k=0,1$ and $l=L,M, H$. 
\begin{lem}[\cite{Ponce}, \cite{Shibata}, \cite{K-S}] \label{Lem:3.1}
	{\rm (i)} Let $1 \le q \le p \le \infty$ with $(p, q) \neq (1,1)$,\ $(\infty, \infty)$, 
	$\ell \ge \tilde{\ell} \ge 0$ and $\alpha \ge \tilde{\alpha} \ge 0$. 
	Then it holds that
	\begin{align} 
	& 
	\sum_{j=1}^{2} \left\|
	\partial_{t}^{\ell} 
	\nabla^{\alpha}
	K_{0L}^{(j)}(t) \ast g
	\right\|_{p} 
	\le C(1+t)^{-\frac{3}{2}(\frac{1}{q}-\frac{1}{p})-(\frac{1}{q}-\frac{1}{p})+\frac{1}{2} -\frac{\ell-\tilde{\ell}+ \alpha-\tilde{\alpha}}{2}}
	\| \nabla^{\tilde{\alpha}+\tilde{\ell}} g \|_{q}, \label{eq:3.1}  \\
	& \sum_{j=1}^{2} \left\| 
	\partial_{t}^{\ell} 
	\nabla^{\alpha}
	K_{1L}^{(j)}(t) \ast g
	\right\|_{p} 
	\le C(1+t)^{-\frac{3}{2}(\frac{1}{q}-\frac{1}{p})-(\frac{1}{q}-\frac{1}{p})+1 -\frac{\ell-\tilde{\ell}+ \alpha-\tilde{\alpha}}{2}}
	\| \nabla^{\tilde{\alpha}+\tilde{\ell}} g \|_{q} \label{eq:3.2} 
	\end{align}
	for $t \ge 0$.
	\\
	{\rm (ii)}
	Let $\alpha \ge \tilde{\alpha} \ge 0$, $\ell \ge 2 \tilde{\ell} \ge 0$ and $t>0$. 
	Then, the following estimates hold:
	\begin{equation}
		\begin{split}
			& \sum_{j=1}^{2} \| \partial_{t}^{\ell} \nabla^{\alpha} K_{0H}^{(j)}(t) \ast g  \|_{p}  \\ 
			& \le C e^{-ct} (\| \nabla^{\alpha_{1}} g \|_{p} +t^{-\frac{3}{2}(\frac{1}{q}-\frac{1}{p})-\frac{\alpha-\tilde{\alpha}}{2}-(\ell-\frac{\tilde{\ell}}{2})+1}
			\| \nabla^{\tilde{\alpha}+\tilde{\ell}} g \|_{q}), \quad \alpha_{1} \ge \alpha, \label{eq:3.3} 
		\end{split}
	\end{equation}
	\begin{equation}
		\begin{split}
			& \sum_{j=1}^{2} \| \partial_{t}^{\ell} \nabla^{\alpha} K_{1H}^{(j)}(t) \ast g \|_{p} \\
			& \le C e^{-ct} (\| \nabla^{\alpha_{1}} g \|_{p} 
			+t^{-\frac{3}{2}(\frac{1}{q}-\frac{1}{p})-\frac{\alpha-\tilde{\alpha}}{2}-(\ell-\frac{\tilde{\ell}}{2})+1}\| \nabla^{\tilde{\alpha}+\tilde{\ell}} g \|_{q}),
			\quad \alpha_{1} \ge \max\{0, \alpha-2 \}
			\label{eq:3.4}
		\end{split}
	\end{equation}
	for $1 < p< \infty$ and $1 \le q \le p$ 
	and
	\begin{equation} \label{eq:3.5}
		\begin{split}
		\sum_{j=1}^{2} \left( \| \partial_{t}^{\ell} \nabla^{\alpha} K_{0M}^{(j)}(t) \ast g \|_{p}
			+\| \partial_{t}^{\ell} \nabla^{\alpha} K_{1M}^{(j)}(t) \ast g \|_{p} \right)
			\le C e^{-ct} \| \nabla^{\tilde{\alpha}} g \|_{q} 
		\end{split}
	\end{equation}
	for $1 \le q \le p \le \infty$. 
\end{lem}
The remainder part of this section summarizes useful inequalities.
The following lemma provides estimate for the Riesz transform:  
\begin{lem}  \label{Lem:3.2}
	Let $1<p<\infty$. 
	There exists $C>0$ such that 
	\begin{equation} \label{eq:3.6}
		\begin{split}
			\|\mathcal{F}^{-1}[R_{j} g]\|_{p} \le C \| g \|_{p}
		\end{split}
	\end{equation}
	for $j=1,2$.
\end{lem}
	For the proof of \eqref{eq:3.6}, see e.g. \cite{Gr}.

The following lemmas are the Hardy-Littlewood-Sobolev inequality and the Sobolev inequality which are useful to show the regularity estimates of the time periodic solutions to \eqref{eq:1.1}.
\begin{lem} \label{Lem:3.3}
	Let $\frac{3}{2} < p< \infty$.
	There exists a constant $C>0$ such that 
	\begin{align} 
			\left\| 
			\mathcal{F}^{-1}[ |\xi|^{-1} \hat{g} ] 
			\right\|_{L^{p}(\R^{3})} 
			\le C \| g \|_{
			L^{\frac{3p}{p+3} }(\R^{3})
			}, \label{eq:3.7} 
	\end{align}
	where $C$ is independent of $g$.
\end{lem}
\begin{lem} \label{Lem:3.4}
	Let $1 < p< 3$.
	There exists a constant $C>0$ such that 
	\begin{align} 
			\left\| 
			g
			\right\|_{L^{\frac{3p}{3-p}}(\R^{3})} 
			\le C \| \nabla g \|_{
			L^{p }(\R^{3})
			}, \label{eq:3.8} 
	\end{align}
	where $C$ is independent of $g$.
\end{lem}
For the proof of lemmas \ref{Lem:3.3}-\ref{Lem:3.4}, see e.g. \cite{Stein1}.

Finally, we recall well-known embedding results.
\begin{lem} \label{Lem:3.5}
	There exists a constant $C>0$ such that 
	\begin{align}
			& \| g \|_{L^{\infty}(\R^{3})} \le C \| g \|_{L^{2}(\R^{3})}^{\frac{1}{4}} \| \nabla^{2} g \|_{L^{2}(\R^{3})}^{\frac{3}{4}}, \label{eq:3.9} \\
			& \| \nabla g \|_{L^{2p}(\R^{3})} \le C \| g \|_{L^{\infty}(\R^{3})}^{\frac{1}{2}} \| \nabla^{2} g \|_{L^{p}(\R^{3})}^{\frac{1}{2}}, \quad 1\le p< \infty, \label{eq:3.10}
	\end{align}
	where $C$ is independent of $g$.
\end{lem}
	The estimate \eqref{eq:3.9}-\eqref{eq:3.10} are proved in \cite{C}.
\section{Estimates for $Q(t) \ast$}
This section is devoted to the estimates for the operation of $Q(t) \ast$.
In what follows, we denote $\alpha_{1} := \sqrt{\lambda+ 2 \mu}$ and $\alpha_{2} := \sqrt{\mu}$ for the simplicity of the notation.
We also introduce the frequency decomposition of $Q(t) \ast$ by 
\begin{equation*} 
\begin{split}
\hat{Q}_{l}(t, \xi) & :=\hat{Q}_{l}^{(1)}(t, \xi) +\hat{Q}_{l}^{(2)}(t, \xi), \\
\hat{Q}_{l}^{(j)}(t, \xi) & := \hat{Q}^{(j)}(t, \xi) \chi_{l}
\end{split}
\end{equation*}
for $j=1,2$ and $l=L,M, H$. 
\subsection{Low frequency parts}
We begin with the asymptotic behavior of the characteristic roots when $|\xi|$ is small.
\begin{lem} \label{Lem:4.1}
Let $j=1,2$. Then it holds that 
\begin{equation*}
\begin{split}
& \sigma_{j,\pm} = \pm i \alpha_{j} |\xi| +O(|\xi|^{2}), \\ 
& \dfrac{1}{
\sigma_{j,+} -\sigma_{j,-}
} = \dfrac{1}{2 i \alpha_{j} |\xi| } +O(|\xi|^{2}), \\
& \dfrac{1}{1-e^{\sigma_{j,\pm} T}} = \mp \dfrac{1}{i \alpha_{j} |\xi| T} + O(|\xi|)
\end{split}
\end{equation*}
as $|\xi| \to 0$.
\end{lem}
Lemma \ref{Lem:4.1} follows from a direct computation by using the formulas of 
$\sigma_{j,\pm}$.

Based on Lemma \ref{Lem:4.1}, 
we establish the estimates of $Q_{L}(t) \ast$, the low frequency parts of $Q(t) \ast$, in the $L^{p}$-Sobolev spaces. 
\begin{prop} \label{prop:4.2}
Let $j=1$, $2$, $\ell \ge 0$, $\alpha \ge 0$ and $0 \le t \le T$. 
Then there exists a constant $C=C_{T}>0$ such that
\begin{equation} \label{eq:4.1}
\begin{split}
\sum_{j=1}^{2} \int_{0}^{T}
\| \nabla^{\alpha} \partial_{t}^{\ell} Q^{(j)}_{L}(t-s) \ast g (s) \|_{p} ds 
\le C \int_{0}^{T} \| g(s) \|_{1}^{\frac{2}{p}-\frac{1}{3}} \| g(s) \|_{2}^{\frac{4}{3}-\frac{2}{p}}ds
\end{split}
\end{equation}
for $\ell+ \alpha=1$ with $\dfrac{3}{2} < p \le 6$
and 
\begin{equation} \label{eq:4.2}
\begin{split}
\sum_{j=1}^{2}
\int_{0}^{T}
\| \nabla^{\alpha} \partial_{t}^{\ell} Q^{(j)}_{L}(t-s) \ast g  \|_{p}ds  \le C \| g \|_{L^{1}(0,T;L^{1})}
\end{split}
\end{equation}
for $\ell+ \alpha=1$ with $2 \le p \le \infty$, 
for $\ell+ \alpha=2$ with $1< p \le \infty$ and for $\ell+ \alpha \ge 3$ with $1 \le p \le \infty$.
\end{prop}
\begin{proof}
We only show the case $j=1$, since the same method also works for the case of $j=2$.  
Observing that 
\begin{equation} \label{eq:4.3}
\begin{split}
\partial_{t}^{\ell} \hat{Q}^{(1)}_{L}(t-s, \xi) = 
\frac{1}{ 2 \alpha_{1}^{2} |\xi|^{2} T } \left\{  
(i \alpha_{1} |\xi|)^{\ell} +(-i \alpha_{1} |\xi|)^{\ell}
\right\} R_{1} \chi_{L}
+ O(|\xi|^{\ell-1})
\end{split}
\end{equation}
by Lemma \ref{Lem:4.1} and $e^{\sigma_{j,\pm} (t+T-s)} =1+O(|\xi|)$ as $|\xi| \to 0$.
Therefore when $\ell+ \alpha=1$ with $\dfrac{3}{2} < p \le 6$, 
we have 
\begin{equation} \label{eq:4.4}
\begin{split}
\int_{0}^{T} \| \nabla^{\alpha} \partial_{t}^{\ell} Q^{(1)}_{L}(t-s) \ast g(s)  \|_{p} ds
\le C \int_{0}^{T}  \| g(s) \|_{\frac{3p}{p+3}} ds
\end{split}
\end{equation}
by \eqref{eq:4.3}, \eqref{eq:3.7} and \eqref{eq:3.6}. 
Moreover the H\"{o}lder inequality yields
\begin{equation} \label{eq:4.5}
\begin{split}
\| g(s) \|_{\frac{3p}{p+3}} \le \| g(s) \|_{1}^{\frac{2}{p}-\frac{1}{3}} \| g(s) \|_{2}^{\frac{4}{3}-\frac{2}{p}},
\end{split}
\end{equation}
since $1< \dfrac{3p}{p+3} \le 2$ for $\dfrac{3}{2} < p \le 6$.
The estimates \eqref{eq:4.4} and \eqref{eq:4.5} show the desired estimate \eqref{eq:4.1}.
When for $\ell+ \alpha=1$ with $2 \le p \le \infty$,
denoting $\frac{1}{p}+\frac{1}{p'}=1$,  
we easily have 
\begin{equation*} 
\begin{split}
\int_{0}^{T} \| \nabla^{\alpha} \partial_{t}^{\ell} Q^{(1)}_{L}(t-s) \ast g(s)  \|_{p} ds
& \le C\int_{0}^{T} \| \xi^{\alpha}  \partial_{t}^{\ell} \hat{Q}^{(1)}_{L}(t-s) \|_{p'} \| \hat{g}(s) \|_{\infty} ds \\
& \le C \| g \|_{L^{1}(0,T;L^{1})}
\end{split}
\end{equation*}
by \eqref{eq:3.6} and the fact that $\frac{1}{|\xi|} \chi_{L} \in L^{1} \cap L^{2}(\R^{3})$, 
which proves the desired estimate \eqref{eq:4.2}.

For the case $\ell+ \alpha=2$ with $1< p < \infty$,
noting \eqref{eq:4.3}, we have 
\begin{equation*} 
\begin{split}
\int_{0}^{T} \| \nabla^{\alpha} \partial_{t}^{\ell} Q^{(1)}_{L}(t-s) \ast g(s)  \|_{p} ds
\le C\int_{0}^{T} \| \mathcal{F}^{-1} [\chi_{L} \hat{g}(s) ] \|_{p} ds \le C \| g \|_{L^{1}(0,T;L^{1})}
\end{split}
\end{equation*}
by \eqref{eq:3.6}, which is the desired conclusion \eqref{eq:4.2}.
When $\ell+ \alpha=2$ with $p=\infty$,
the direct calculation with \eqref{eq:4.3} shows  
\begin{equation*} 
\begin{split}
\int_{0}^{T} \| \nabla^{\alpha} \partial_{t}^{\ell} Q^{(1)}_{L}(t-s) \ast g(s)  \|_{\infty} ds
\le C \int_{0}^{T} \| \chi_{L} \hat{g}(s) \|_{1} ds \le C \| g \|_{L^{1}(0,T;L^{1})},
\end{split}
\end{equation*}
which is the desired estimate \eqref{eq:4.2}.
The remainder part $\ell+ \alpha \ge 3$ with $1 \le p \le \infty$ is more easier than the previous cases.
We omit the detail. 
We complete the proof of Proposition \ref{prop:4.2}.
\end{proof}
\subsection{Middle and High frequency parts}
We collect the asymptotic behavior of the characteristic roots as $|\xi| \to \infty$.
\begin{lem} \label{Lem:4.3}
Let $j=1,2$. Then it holds that 
\begin{equation*}
\begin{split}
& \sigma_{j,+} = -\frac{\alpha_{j}^{2}}{\nu} +O(|\xi|^{-2}), \\ 
& \sigma_{j,-} = -\nu |\xi|^{2}  +O(1), \\ 
& \dfrac{1}{
\sigma_{j,+} -\sigma_{j,-}
} = \dfrac{1}{\nu |\xi|^{2} } +O(|\xi|^{-4}), \\
& \dfrac{1}{1-e^{\sigma_{j,+} T}} = \dfrac{1}{1-e^{-\frac{\alpha_{j}^{2}}{\nu}T}} + O(|\xi|^{-2}), \\
& \dfrac{1}{1-e^{\sigma_{j,-} T}} = 1+O(e^{-|\xi|^{2}})
\end{split}
\end{equation*}
as $|\xi| \to \infty$.
\end{lem}
Lemma \ref{Lem:4.3} follows from a direct computation by using the formulas of $\sigma_{j,\pm}$. 
%

The following proposition is the estimates of $Q_{H}(t) \ast$, the high frequency parts of $Q(t) \ast$, in the $L^{p}$-Sobolev spaces,
which suggests that $Q_{H}(t) \ast$ behaves like $K_{1H}(t) \ast$.
\begin{prop} \label{prop:4.4}
Let $\alpha \ge 0$, $\ell \ge 0$ and $t>0$. 
	Then, there exists a constant $C=C_{T}>0$ such that
	\begin{equation} \label{eq:4.6}
		\begin{split}
			& \sum_{j=1}^{2} \int_{0}^{T} 
			\| \partial_{t}^{\ell} \nabla^{\alpha} Q_{H}^{(j)}(t-s) \ast g(s) \|_{p}  ds \le C \displaystyle\int_{0}^{T} \| \nabla^{\tilde{\alpha}} g(s) \|_{p} ds
		\end{split}
	\end{equation}
	for $\tilde{\alpha} \ge \max\{0, \alpha-2,  2 \ell + \alpha-2 \}$, 
	$1 < p< \infty$ and $1 \le q \le p$ 
	and
	\begin{equation} \label{eq:4.7}
		\begin{split}
			\sum_{j=1}^{2} \int_{0}^{T}
			 \| \partial_{t}^{\ell} \nabla^{\alpha} Q_{M}^{(j)}(t-s) \ast g \|_{p} ds
			\le C \int_{0}^{T} \| \nabla^{\tilde{\alpha}} g(s) \|_{q} ds 
		\end{split}
	\end{equation}
	for $1 \le q \le p \le \infty$ and $\tilde{\alpha} \ge 0$.
\end{prop}
%
%
\begin{proof}
We only show the estimate \eqref{eq:4.6} for $j=1$.
The case \eqref{eq:4.6} for $j=2$ is shown by the same way.
The estimate \eqref{eq:4.7} is easily verified by the expression of $\hat{Q}_{M}^{(j)}$, since the support of $\hat{Q}_{M}^{(j)}$
is bounded and does not include the neighborhood of the origin in the Fourier space.
To show the estimate \eqref{eq:4.6}, 
we observe that  
\begin{equation*}
\begin{split}
\partial_{t}^{\ell} \hat{Q}^{(1)}_{H}(t-s, \xi) & = 
\left( -\frac{\alpha_{j}^{2}}{\nu} \right)^{\ell}
\dfrac{
e^{-\frac{\alpha_{j}^{2}}{\nu}(t-s+T)} \chi_{H}
}
{
\nu |\xi|^{2} \left( 1-e^{-\frac{\alpha_{j}^{2}}{\nu}T} \right) 
} R_{1}
+O(|\xi|^{-4}) \\
& +(-\nu |\xi|^{2} )^{\ell}
\dfrac{
e^{-\nu |\xi|^{2} (t-s+T)} \chi_{H}
}
{
\nu |\xi|^{2} 
} R_{1} + e^{-\nu |\xi|^{2} (t-s+T)} O(|\xi|^{2(\ell-2)}) 
\end{split}
\end{equation*}
as $|\xi| \to \infty$ for $\ell \ge 0$ by Lemma \ref{Lem:4.3}.
Applying \eqref{eq:3.6}, 
we see that 
\begin{equation}  \label{eq:4.8}
\begin{split}
& \int_{0}^{T} \| \nabla^{\alpha} \partial_{t}^{\ell} Q^{(1)}_{H}(t-s) \ast g(s) \|_{p}ds  \\
& \le C \int_{0}^{T} \| \mathcal{F}^{-1}[ \chi_{H} |\xi|^{\alpha-2} \hat{g}(s)] \|_{p}
+  \| \mathcal{F}^{-1}[ \chi_{H} |\xi|^{2 \ell+\alpha-2} e^{-\nu |\xi|^{2} (t-s+T)}  \hat{g}(s)] \|_{p} ds.
\end{split}
\end{equation}
When $2 \ell +\alpha-2 \ge 0$, 
we have 
\begin{equation}  \label{eq:4.9}
\begin{split}
& \int_{0}^{T}  \| \mathcal{F}^{-1}[ \chi_{H} |\xi|^{2 \ell+\alpha-2} e^{-\nu |\xi|^{2} (t-s+T)}  \hat{g}(s)] \|_{p} ds \\
& \le C \int_{0}^{T}  \| \mathcal{F}^{-1}[ \chi_{H}  e^{-\nu |\xi|^{2} (t-s+T)} ] \|_{1} 
\| \nabla^{2 \ell+\alpha-2} g(s) \|_{p} ds \\
& \le C \int_{0}^{T}  \| \mathcal{F}^{-1}[e^{-\nu |\xi|^{2} (t-s+T)} ] \|_{1} 
\| \nabla^{2 \ell+\alpha-2} g(s) \|_{p} ds \\
& + C \int_{0}^{T}  \| \mathcal{F}^{-1}[(\chi_{L}+\chi_{M})e^{-\nu |\xi|^{2} (t-s+T)} ] \|_{1} 
\| \nabla^{2 \ell+\alpha-2} g(s) \|_{p} ds \\
& \le C \int_{0}^{T} 
\| \nabla^{2 \ell+\alpha-2} g(s) \|_{p} ds, 
\end{split}
\end{equation}
where we used the fact that $\chi_{H}=1-\chi_{L}-\chi_{M}$.
In contrast, if $2 \ell +\alpha -2< 0$ and $\tilde{\alpha} \ge \max\{ \alpha-2,0 \}$,
we see that
\begin{equation}  \label{eq:4.10}
\begin{split}
& \int_{0}^{T}  \| \mathcal{F}^{-1}[ \chi_{H} |\xi|^{2 \ell+\alpha-2} e^{-\nu |\xi|^{2} (t-s+T)}  \hat{g}(s)] \|_{p} ds \\
& \le C \int_{0}^{T}  \| \mathcal{F}^{-1}[ \chi_{H}  |\xi|^{2 \ell+\alpha-2-\tilde{\alpha}} e^{-\nu |\xi|^{2} (t-s+T)} ] \|_{1} 
\| \nabla^{\tilde{\alpha}} g(s) \|_{p} ds \\
& \le C \int_{0}^{T}  (T+t-s)^{-\ell-\frac{\alpha}{2}+1+\frac{\tilde{\alpha}}{2}}
\| \nabla^{\tilde{\alpha}} g(s) \|_{p} ds \\
& \le C \int_{0}^{T} 
\| \nabla^{\tilde{\alpha}} g(s) \|_{p} ds,
\end{split}
\end{equation}
where we used the fact that
\begin{equation*}
\begin{split}
(T+t-s)^{-\ell-\frac{\alpha}{2}+1+\frac{\tilde{\alpha}}{2}} \le C_{T}
\end{split}
\end{equation*}
for $0 \le s,t \le T$,
since $-\ell-\frac{\alpha}{2}+1+\frac{\tilde{\alpha}}{2} \ge 0$.
Combining the estimates \eqref{eq:4.8}-\eqref{eq:4.10}, 
we obtain the desired estimate \eqref{eq:4.6}.
\end{proof}
%
%
\section{Existence and Regularity}
In this section, we give the proof of Theorems \ref{thm:2.1} and \ref{thm:2.3}.
\subsection{Proof of the existence of time periodic solution}
First, we define the solution space as 
\begin{equation*}
\begin{split}
X_{1}:= \left\{ 
u \in \{ C([0,T]; \dot{W}^{3,p_{0}} \cap \dot{H}^{3} \cap \dot{H}^{1}) 
\cap C^{1}([0,T]; W^{1,p_{0}} \cap H^{1} ) \}^{3} ; \| u \|_{X_{1}} < \infty \right\}
\end{split}
\end{equation*}
with the norm 
\begin{equation*}
\begin{split}
\| u \|_{X_{1}} := \sup_{t \in [0,T]} \left\{ 
\sum_{q= p_{0},2} (\| \nabla^{3} u_{per}(t) \|_{q} +\| \nabla \partial_{t} u_{per}(t) \|_{q} 
+\| \partial_{t} u_{per}(t) \|_{q} )+\| \nabla u_{per}(t) \|_{2}
\right\}. 
\end{split}
\end{equation*}
Noting the integral equation \eqref{eq:2.5},  
we define the mapping $\Phi_{1}$ on $\mathbb{B}_{R}:= \{ u \in X_{1}; \| u \|_{X_{1}} \le R \}$, 
the ball with a radius $R>0$ in $X_{1}$, 
where $R$ is determined later:
\begin{equation*} 
\begin{split}
\Phi_{1}[u] (t) & := I_{1}[u](t)+J_{1}[g](t),
\end{split}
\end{equation*}
where
\begin{equation*}
\begin{split}
I_{1}[u](t) & := \int_{0}^{T} Q(t-s) \ast F(u)(s) ds + \int_{0}^{t} K_{1}(t-s) \ast F(u)(s) ds, \\
J_{1}[g](t) & := \int_{0}^{T} Q(t-s) \ast g(s) ds + \int_{0}^{t} K_{1}(t-s) \ast g(s) ds.
\end{split}
\end{equation*}
We claim that $\Phi_{1}$ is a contraction mapping on $\mathbb{B}_{R}:= \{ u \in X_{1}; \| u \|_{X_{1}} \le R \}$, 
if $R>0$ is sufficiently small.
Namely, we prove that the estimates 
\begin{equation} \label{eq:5.1}
\begin{split}
\| \Phi_{1}[u]  \|_{X_{1}} \le R
\end{split}
\end{equation}
and 
\begin{equation} \label{eq:5.2}
\begin{split}
\| \Phi_{1}[u] - \Phi_{1}[v] \|_{X_{1}} \le \frac{1}{2} \| u-v \|_{X_{1}}
\end{split}
\end{equation}
for $u,v \in \mathbb{B}_{R}$.
Once we have the estimates \eqref{eq:5.1} and \eqref{eq:5.2},
we conclude the unique existence of the periodic solution satisfying the integral equation \eqref{eq:2.5}
by the Banach fixed point theorem.

First, we estimate the nonlinear terms in $X_{1}$ as in \cite{K-T1} and \cite{K-T2} by \eqref{eq:3.9} and \eqref{eq:3.10}:
\begin{equation} \label{eq:5.3}
\begin{split}
\| F(u)(s) \|_{1} \le C \| \nabla u(s) \|_{2} \| \nabla^{2} u(s) \|_{2} \le C \| \nabla u(s) \|_{2}^{\frac{3}{2}} \| \nabla^{3} u(s) \|_{2}^{\frac{1}{2}}  \le C \| u \|_{X_{1}}^{2},
\end{split}
\end{equation}
\begin{equation} \label{eq:5.4}
\begin{split}
\| F(u)(s) \|_{2} 
\le C \| \nabla u(s) \|_{\infty} \| \nabla^{2} u(s) \|_{2} 
\le C \| \nabla u(s) \|_{2}^{\frac{3}{4}} \| \nabla^{3} u(s) \|_{2}^{\frac{5}{4}} 
\le C \| u \|_{X_{1}}^{2}
\end{split}
\end{equation}
and
\begin{equation} \label{eq:5.5}
\begin{split}
\| \nabla F(u)(s) \|_{q} & \le C \| \nabla^{2} u(s) \|_{2q}^{2} + \| \nabla  u\nabla^{3} u(s) \|_{q} 
\le C \| \nabla u(s) \|_{\infty} \| \nabla^{3} u(s) \|_{q}  \\
&  \le C \| \nabla u(s) \|_{2}^{\frac{1}{4}} \| \nabla^{3} u(s) \|_{2}^{\frac{3}{4}} \| \nabla^{3} u(s) \|_{q} \le C \| u \|_{X_{1}}^{2}
\end{split}
\end{equation}
with $q=2$ and $p_{0}$.
For the simplicity of the notation, we denote 
\begin{equation*}
\begin{split}
I_{1}[u](t) =I_{11}[u](t)+I_{12}[u](t),
\end{split}
\end{equation*}
where
\begin{equation*}
\begin{split}
I_{11}[u](t) & := \int_{0}^{T} Q(t-s) \ast F(u)(s) ds, \\
I_{12}[u](t) & := \int_{0}^{t} K_{1}(t-s) \ast F(u)(s) ds.
\end{split}
\end{equation*}
We begin with the estimates for $I_{11}[u](t)$.
Applying the estimates \eqref{eq:4.2}, \eqref{eq:4.6}, \eqref{eq:4.7}, \eqref{eq:5.3} \and \eqref{eq:5.4}, 
we have
\begin{equation} \label{eq:5.6}
\begin{split}
\left\|
\nabla I_{11}[u](t)
\right\|_{2} 
& \le
\sum_{l=L,M,H}
\int_{0}^{T}
 \left\| \nabla Q_{l}(t-s) \ast F(u)(s) \right\|_{2}
ds \\
& \le C
\int_{0}^{T}
 \left\|  F(u)(s) \right\|_{1} +\left\|  F(u)(s) \right\|_{2}
ds 
\le C \|u \|_{X_{1}}^{2}. 
\end{split}
\end{equation}
Similarly, we apply the estimates \eqref{eq:4.2}, \eqref{eq:4.6}, \eqref{eq:4.7}, \eqref{eq:5.3} and \eqref{eq:5.5} to see that  
\begin{equation} \label{eq:5.7}
\begin{split}
& \sum_{q=p_{0},2}\left\|
\nabla^{3} I_{11}[u](t)
\right\|_{q} + \sum_{q=p_{0},2} \sum_{\alpha=0,1}\left\|
\nabla^{\alpha} \partial_{t} I_{11}[u](t)
\right\|_{q} \\
& \le  \sum_{q=p_{0},2} \sum_{l=L,M,H}
\int_{0}^{T}
 \left\| \nabla^{3} Q_{l}(t-s) \ast F(u)(s) \right\|_{q}
ds \\
& + \sum_{q=p_{0},2} \sum_{\alpha=0,1} \sum_{l=L,M,H}
\int_{0}^{T}
 \left\| \nabla^{\alpha} \partial_{t} Q_{l}(t-s)\ast F(u)(s) \right\|_{q}
ds 
 \\
& \le C
\int_{0}^{T}
 \left\|  F(u)(s) \right\|_{1}
ds 
+ C\int_{0}^{T} 
( \left\| \nabla F(u)(s) \right\|_{p_{0}} +\left\| \nabla F(u)(s) \right\|_{2} )
ds 
 \\
& \le C \|u \|_{X_{1}}^{2}.
\end{split}
\end{equation}
Secondly, we estimate $I_{12}[u]$.
Here we note that the estimates of $I_{12}[u]$ in the $L^{2}$-Sobolev space are shown in \cite{K-T1}.
Namely, we have
\begin{equation} \label{eq:5.8}
\begin{split}
& \left\|
\nabla^{\alpha} I_{12}[u](t)
\right\|_{2} \le C(1+t)^{-\frac{1}{4}-\frac{\alpha}{2} } \| u \|_{X_{1}}^{2}
\le C \|u \|_{X_{1}}^{2}
\end{split}
\end{equation}
for $1 \le \alpha \le 3$ and 
\begin{equation} \label{eq:5.9}
\begin{split}
& \left\|
\nabla^{\alpha} \partial_{t} I_{12}[u](t)
\right\|_{2} \le C(1+t)^{-\frac{3}{4}-\frac{\alpha}{2} } \| u \|_{X_{1}}^{2}
\le C \|u \|_{X_{1}}^{2}
\end{split}
\end{equation}
for $0 \le \alpha \le 1$.
The remainder part of the estimate of $I_{12}[u]$ is estimated as follows:
\begin{equation} \label{eq:5.10}
\begin{split}
\left\|
\nabla^{3} I_{12}[u](t)
\right\|_{p_{0}}
& \le \sum_{l=L,M,H}
\int_{0}^{t}
 \left\| \nabla^{3} K_{1l}(t-s) \ast F(u)(s) \right\|_{p_{0}}
ds \\
& \le C
\int_{0}^{t}(1+t-s)^{-\frac{5}{2}(1-\frac{1}{p_{0}})- \frac{1}{2} }
( \left\| F(u)(s) \right\|_{1} +\left\| \nabla F(u)(s) \right\|_{p_{0}} )
ds \\
& \le C\|u \|_{X_{1}}^{2}
\int_{0}^{t}
ds 
\le C_{T} \|u \|_{X_{1}}^{2}
\end{split}
\end{equation}
and 
\begin{equation} \label{eq:5.11}
\begin{split}
& \sum_{q=p_{0},2} \sum_{\alpha=0,1}\left\|
\nabla^{\alpha} \partial_{t} I_{12}[u](t)
\right\|_{q} \\
& \le \sum_{q=p_{0},2} \sum_{\alpha=0,1} \sum_{l=L,M,H}
\int_{0}^{t}
 \left\| \nabla^{\alpha} \partial_{t} K_{1l}(t-s)\ast F(u)(s) \right\|_{q}
ds 
 \\
& \le C \sum_{q=p_{0},2} \sum_{\alpha=0,1}
\int_{0}^{t}(1+t-s)^{-\frac{5}{2}(1-\frac{1}{q})+1- \frac{\alpha+1}{2} }
( \left\| F(u)(s) \right\|_{1} +\left\| \nabla F(u)(s) \right\|_{q} )
ds \\
& \le C\|u \|_{X_{1}}^{2}
\int_{0}^{t}
ds 
\le C_{T} \|u \|_{X_{1}}^{2}
\end{split}
\end{equation}
by \eqref{eq:3.2}, \eqref{eq:3.4}, \eqref{eq:3.5}, \eqref{eq:5.3} and \eqref{eq:5.5}.
Using the same argument, we also have 
\begin{equation} \label{eq:5.12}
\begin{split}
& \sum_{q=p_{0},2}
 \left\|
\nabla^{3} J_{1}[g](t)
\right\|_{q} +\left\|
\nabla J_{1}[g](t)
\right\|_{2} +\sum_{q=p_{0},2} 
 \sum_{\alpha=0,1}
\left\|
\nabla^{\alpha} \partial_{t} J_{1}[u](t)
\right\|_{2}  \\
& 
\le C \|g \|_{L^{1}(0,T; \dot{W}^{1,p_{0}} \cap \dot{H}^{1} \cap L^{1} )}.
\end{split}
\end{equation}
Combining \eqref{eq:5.6}-\eqref{eq:5.12},
we arrive at the estimate 
\begin{equation*}
\begin{split}
\| \Phi_{1}[u] \|_{X_{1}} \le C_{0} \| u \|_{X_{1}}^{2} + C_{1} \|g \|_{L^{1}(0,T; \dot{W}^{1,p_{0}} \cap \dot{H}^{1} \cap L^{1} )},
\end{split}
\end{equation*}
where $C_{0}>0$ and $C_{1}>0$.
This implies that choosing $R= 2C_{1} \|g \|_{L^{1}(0,T; \dot{W}^{1,p_{0}} \cap \dot{H}^{1} \cap L^{1} )}$ and $\tilde{\varepsilon}_{0} =\frac{1}{4 C_{0} C_{1}}$,
we get the desired estimate \eqref{eq:5.1} for $0< \varepsilon_{0} < \tilde{\varepsilon}_{0}$.
Again, we apply the same argument to deduce that  
\begin{equation*}
\begin{split}
\| \Phi_{1}[u] -\Phi_{1}[v]  \|_{X_{1}} \le C_{2} R \| u-v \|_{X_{1}}.
\end{split}
\end{equation*}
Namely, taking $\varepsilon_{0}' =\frac{1}{4 C_{1} C_{2}}$,
we have the estimate \eqref{eq:5.2} for $0< \varepsilon_{0} < \varepsilon_{0}'$.
Therefore we have a unique time periodic solution $u_{per}$ in 
\begin{equation*} 
Z:= \{ C([0,T]; \dot{W}^{3,p_{0}} \cap \dot{H}^{3} \cap \dot{H}^{1}) 
\cap C^{1}([0,T]; W^{1, p_{0}} \cap H^{1} ) \}^{3} 
\end{equation*}
satisfying the estimate \eqref{eq:2.6} with 
$\varepsilon_{0}:= \min \{ \tilde{\varepsilon}_{0} ,\varepsilon_{0}' \}$. 
Once we have a unique time periodic solution in $Z$ with \eqref{eq:2.6},
the argument in \cite{K-T} immediately yields the uniqueness of $u_{per}$ in  
$$
\{ C_{per}(\R; \dot{W}^{3,p_{0}} \cap \dot{H}^{3} \cap \dot{H}^{1} ) 
\cap C_{per}^{1}(\R; W^{1, p_{0}} \cap H^{1} ) \}^{3}.
$$
This completes the proof of Theorem \ref{thm:2.1}.
\subsection{Regularity estimates of the periodic solutions}
In this subsection, we give a proof of Theorem \ref{thm:2.3}.
For simplicity of notation, we let $u$ stand for $u_{per}$.
We firstly prove the estimate \eqref{eq:2.12}.
When $2 \le p \le \infty$, 
we have 
\begin{equation*}
\begin{split}
\| \nabla u(t) \|_{p} \le \| \nabla u(t) \|_{\infty}^{1-\frac{2}{p}} \| \nabla u(t) \|_{2}^{\frac{2}{p}} 
\le C \| \nabla^{3} u(t) \|_{2}^{\frac{3}{4}(1-\frac{2}{p})} \| \nabla u(t) \|_{2}^{\frac{2}{p}+\frac{1}{4}(1-\frac{2}{p})} \le C \varepsilon_{0}
\end{split}
\end{equation*}
by \eqref{eq:3.1} and \eqref{eq:3.9}, which is the desired result.
When $\frac{3}{2} < p < 2$, 
we see that 
\begin{equation} \label{eq:5.13}
\begin{split}
\| \nabla u(t) \|_{p} \le \| \nabla I_{11}[u](t) \|_{p} + \| \nabla I_{12}[u](t) \|_{p} + \| \nabla J_{1}[g](t) \|_{p}
\end{split}
\end{equation}
by the integral equation \eqref{eq:2.5}.

For the estimate of $\| \nabla I_{11}[u](t) \|_{p}$ with $\frac{3}{2}< p<2$, 
we apply the estimates \eqref{eq:4.1}, \eqref{eq:4.6}, \eqref{eq:4.7}, \eqref{eq:5.3} and \eqref{eq:5.4} to see that 
\begin{equation} \label{eq:5.14}
\begin{split}
& \| \nabla I_{11}[u](t) \|_{p} \\
& \le C \int_{0}^{T} \| F(u)(s)\|_{1}^{\frac{2}{p}-\frac{1}{3}} \| F(u)(s)\|_{2}^{\frac{4}{3}-\frac{2}{p}} 
+ \| F(u)(s) \|_{1}^{\frac{2}{p}-1} \| F(u)(s) \|_{2}^{2-\frac{2}{p}} ds \\
& \le C_{T} \varepsilon_{0}, 
\end{split}
\end{equation}
where we used the fact that 
\begin{equation} \label{eq:5.15}
\begin{split}
\| f \|_{p} \le \| f \|_{1}^{\frac{2}{p}-1} \| f \|_{2}^{2-\frac{2}{p}}
\end{split}
\end{equation}
for $1\le p \le 2$ with taking $f=F(u(s))$.
Next, we show the estimate for $\| \nabla I_{12}[u](t) \|_{p}$ with $\frac{3}{2}< p<2$.
The estimates \eqref{eq:3.2}, \eqref{eq:3.4}, \eqref{eq:3.5}, \eqref{eq:5.15} with $f=F(u(s))$, \eqref{eq:5.3} and \eqref{eq:5.4} lead to 
\begin{equation} \label{eq:5.16}
\begin{split}
\| \nabla I_{12}[u](t) \|_{p} \le C \int_{0}^{t} (1+t-s)^{\frac{1}{2}} \| F(u)(s) \|_{1}^{\frac{2}{p}-1} \| F(u)(s) \|_{2}^{2-\frac{2}{p}} ds
 \le C_{T} \varepsilon_{0}.
\end{split}
\end{equation} 
By the same way, we also have 
\begin{equation} \label{eq:5.17}
\begin{split}
\| \nabla J_{1}[u](t) \|_{p} \le C_{T} \| g \|_{L^{1}(0,T;L^{1} \cap \dot{H}^{1} \cap \dot{W}^{1,p_{0}})}
 \le C_{T} \varepsilon_{0} 
\end{split}
\end{equation} 
based on the fact that $L^{1} \cap \dot{H}^{1} \cap \dot{W}^{1,p_{0}} = L^{1} \cap \dot{H}^{1} \cap \dot{W}^{1,p_{0}} \cap L^{2}$.

Summing up the estimates \eqref{eq:5.13}, \eqref{eq:5.14}, \eqref{eq:5.16} and \eqref{eq:5.17}, 
we obtain the estimate \eqref{eq:2.12} for $\frac{3}{2} < p <2$.

Secondly we prove the estimate \eqref{eq:2.13}. 
The case $2 \le p \le p_{\ast}$ is shown by the interpolation between $\dot{H}^{2}$ and $\dot{W}^{3, p_{0}}$.
We omit the proof. 
For $1< p \le 2$, 
we use the estimates \eqref{eq:4.2}, \eqref{eq:4.6}, \eqref{eq:4.7}, \eqref{eq:5.15} with $f=F(u(s))$, \eqref{eq:5.3} and \eqref{eq:5.4} to show  
\begin{equation} \label{eq:5.18}
\begin{split}
\| \nabla^{2} I_{11}[u](t) \|_{p} 
& \le C \int_{0}^{T} \| F(u)(s)\|_{1}
+ \| F(u)(s) \|_{1}^{\frac{2}{p}-1} \| F(u)(s) \|_{2}^{2-\frac{2}{p}} ds  \le C_{T} \varepsilon_{0}. 
\end{split}
\end{equation}
We also see that 
\begin{equation} \label{eq:5.19}
\begin{split}
\| \nabla^{2} I_{12}[u](t) \|_{p} +\| \nabla^{2} J_{1}[g](t) \|_{p} \le C_{T} \varepsilon_{0}
\end{split}
\end{equation}
as in the proof of \eqref{eq:5.16} and \eqref{eq:5.18}.
Recalling \eqref{eq:2.5}, 
we obtain the desired estimate \eqref{eq:2.13} for $1<p<2$ 
by \eqref{eq:5.18} and \eqref{eq:5.19}.
We show the estimate \eqref{eq:2.14}.
When $2 \le p \le p_{0}$, 
the desired estimate follows from the interpolation between $\dot{H}^{3}$ and $\dot{W}^{3, p_{0}}$.
For $1 \le p \le 2$,
we apply the same method of the proof of \eqref{eq:2.13}, with replaced by \eqref{eq:5.15} with $f=\nabla F(u(s))$,
to obtain the desired conclusion. 
The remainder part, the case $p_{0}=3$, is shown in the same way. 
We complete the proof of Theorem \ref{thm:2.3}.
%
\section{Stability}
In this section, we prove Theorem \ref{thm:2.2}.
At first, we formulate the Cauchy problem \eqref{eq:2.7}-\eqref{eq:2.8} into the following integral equation on $\tilde{u}$: 
\begin{equation} \label{eq:6.1}
\begin{split}
\tilde{u}(t) = \tilde{u}_{lin}(t)+ I_{2}[\tilde{u}](t)+ I_{3}[\tilde{u}](t)+ I_{4}[\tilde{u}](t)
\end{split}
\end{equation}
by the Duhamel principle, where
\begin{equation*}
\begin{split}
 \tilde{u}_{lin}(t) & := K_{0}(t) \ast f_{0} +K_{1}(t) \ast f_{1}, \\
 I_{2}[\tilde{u}](t)& := \int_{0}^{t} K_{1}(t-s) \ast F(\tilde{u}(s)) ds, \\
 I_{3}[\tilde{u}](t) & := \int_{0}^{t} K_{1}(t-s) \ast (\nabla \tilde{u}(s) \nabla^{2} u_{per}(s) )ds, \\
 I_{4}[\tilde{u}](t) & := \int_{0}^{t} K_{1}(t-s) \ast (\nabla u_{per}(s) \nabla^{2} \tilde{u}(s)) ds.
\end{split}
\end{equation*}
In what follows, we use the notation 
\begin{equation*}
\begin{split}
 I_{3l}[\tilde{u}](t) & := \int_{0}^{t} K_{1l}(t-s) \ast (\nabla \tilde{u}(s) \nabla^{2} u_{per}(s)) ds, \\
 I_{4l}[\tilde{u}](t) & := \int_{0}^{t} K_{1l}(t-s) \ast (\nabla u_{per}(s) \nabla^{2} \tilde{u}(s)) ds
\end{split}
\end{equation*}
for $l=L,M,H$.
%
%
%
%
%
%
%
%
%
%
Introducing the function space $X_{2}$ by 
\begin{equation*}
\begin{split}
X_{2}:= \left\{ 
u \in C([0,\infty); \dot{W}^{3, \frac{5}{2}} \cap \dot{H}^{3} \cap \dot{H}^{1}) \cap C^{1}([0,\infty); W^{1, \frac{5}{2}} \cap H^{1}) ; \| u \|_{X_2}< \infty \right\}
\end{split}
\end{equation*}
with the norm 
\begin{equation*}
\begin{split}
\| u \|_{X_2} :=  \sup_{t \in [0,\infty)} & \left\{ 
(1+t)^{\frac{3}{4}}\| \nabla u(t) \|_{2} +(1+t)^{\frac{7}{4}}\| \nabla^{3} u(t) \|_{2} ++(1+t)^{2}\| \nabla^{3} u(t) \|_{\frac{5}{2}} 
\right. \\
& + \left. (1+t)^{2} \| \nabla u (t) \|_{\infty}  +(1+t)^{\frac{5}{4}}\| \nabla \partial_{t} u(t) \|_{2} +(1+t)^{\frac{3}{4}}\| \partial_{t} u(t) \|_{2} \right. \\
& + \left. (1+t)^{\frac{3}{2}} \| \nabla \partial_{t} u (t) \|_{\frac{5}{2}} 
+(1+t)\| \partial_{t} u(t) \|_{\frac{5}{2}} \right\}, 
\end{split}
\end{equation*}
we define the mapping $\Phi_{2}$ by 
\begin{equation*}
\begin{split}
\Phi_{2}[\tilde{u}](t) := \tilde{u}_{lin}(t)+ I_{2}[\tilde{u}](t)+ I_{3}[\tilde{u}](t)+ I_{4}[\tilde{u}](t)
\end{split}
\end{equation*}
on $\tilde{\mathbb{B}}_{R}:= \{ u \in X_{2}; \| u \|_{X_{2}} \le R \}$.
As in the previous section, we need to prove the estimates 
\begin{equation} \label{eq:6.2}
\begin{split}
\| \Phi_{2}[\tilde{u}]  \|_{X_{2}} \le R
\end{split}
\end{equation}
and 
\begin{equation} \label{eq:6.3}
\begin{split}
\| \Phi_{2}[\tilde{u}] - \Phi_{2}[\tilde{v}] \|_{X_{2}} \le \frac{1}{2} \| \tilde{u}-\tilde{v} \|_{X_{2}}
\end{split}
\end{equation}
for $\tilde{u},\tilde{v} \in \tilde{\mathbb{B}}_{R}$,
where $R>0$ is determined later.
Here we note that the estimates of $\tilde{u}_{lin}(t)$ and $I_{2}[\tilde{u}](t)$ in $X_{2}$ immediately follow from the  method in \cite{K-T1} and \cite{K-T2} by applying the estimates \eqref{eq:3.1}-\eqref{eq:3.5} and \eqref{eq:5.3}-\eqref{eq:5.5}:
\begin{align}
& \| \tilde{u}_{lin} \|_{X_{2}} \le C_{0}  \| f_{0}, f_{1} \|_{Y_{1}}, \label{eq:6.4}\\
& \| I_{2}[\tilde{u}] \|_{X_{2}} \le C_{2} \| \tilde{u} \|_{X_{2}}^{2}, \label{eq:6.5} \\
& \| I_{2}[\tilde{u}] -  I_{2}[\tilde{v}] \|_{X_{2}} \le \tilde{C}_{2} R \| \tilde{u} - \tilde{v} \|_{X_{2}} \label{eq:6.6}
\end{align}
for $\tilde{u}$, 
$\tilde{v}  \in \tilde{\mathbb{B}}_{R}$ and $C_{0}$, $C_{2}$, $\tilde{C}_{2}>0$.
The proof of Theorem \ref{thm:2.2} is completed if we prove the estimates 
\begin{align}
& \| I_{j}[\tilde{u}] \|_{X_{2}} \le C_{j} \varepsilon_{0} \| \tilde{u} \|_{X_{2}} , \label{eq:6.7} \\
& \| I_{j}[\tilde{u}] -  I_{j}[\tilde{v}] \|_{X_{2}} \le \tilde{C}_{j} \varepsilon_{0} \| \tilde{u} - \tilde{v} \|_{X_{2}} \label{eq:6.8}
\end{align}
for $\tilde{u}$, 
$\tilde{v}  \in \tilde{\mathbb{B}}_{R}$ and $C_{j}$, $\tilde{C}_{j}>0$ with $j=3,4$.
Indeed, once we have the estimate \eqref{eq:6.7},
choosing 
\begin{equation*} 
\begin{split}
R & := 3 C_{0}  \| f_{0}, f_{1} \|_{Y_{1}}+ 3(C_{3}+C_4{}) \varepsilon_{0}, \\
\tilde{\varepsilon}_{1} &:= \frac{1}{\max\{ 3(C_{3}+C_{4}), 9C_{2}(C_{0}+C_{3} +C_{4}) \}}, \\
\varepsilon_{1}' & := \frac{1}{2 \max \{ 3 C_{0} \tilde{C}_{2}, 3(C_{3} +C_{4})+\tilde{C}_{3} +\tilde{C}_{4} \} },
\end{split}
\end{equation*}
we see that 
\begin{equation*} 
\begin{split}
\| \Phi_{2}[\tilde{u}]  \|_{X_{2}}  
& \le C_{0}  \| f_{0}, f_{1} \|_{Y_{1}} + C_{2} \| \tilde{u} \|_{X_{2}}^{2}
+(C_{3} +C_{4})\varepsilon_{0} \| \tilde{u} \|_{X_{2}} \\
& \le \frac{R}{3} + C_{2} R^{2} + (C_{3} +C_{4})\varepsilon_{0} R \le R
\end{split}
\end{equation*}
by \eqref{eq:6.4}, \eqref{eq:6.5}
for $0< \varepsilon_{1} \le \tilde{\varepsilon}_{1}$ and 
$\| f_{0}, f_{1} \|_{Y_{1}} + \varepsilon_{0} < \varepsilon_{1}$, which proves \eqref{eq:6.2}.
Moreover, taking  
\begin{equation*} 
\begin{split}
\varepsilon_{1}' & := \frac{1}{2 \max \{ 3 C_{0} \tilde{C}_{2}, 3(C_{3} +C_{4})+\tilde{C}_{3} +\tilde{C}_{4} \} },
\end{split}
\end{equation*}
we use the estimates \eqref{eq:6.6} and \eqref{eq:6.8} to have  
\begin{equation*} 
\begin{split}
\| \Phi_{2}[\tilde{u}] - \Phi_{2}[\tilde{v}] \|_{X_{2}} & \le 
\left(\tilde{C}_{2} R + (\tilde{C}_{3}+\tilde{C}_{4}) \varepsilon_{0} \right) 
\| \tilde{u} -\tilde{v} \|_{X_{2}} \\
&= \left\{ 3C_{0} \tilde{C}_{2} \| f_{0}, f_{1} \|_{Y_{1}} +\left( 
3(C_{3} +C_{4})+\tilde{C}_{3} +\tilde{C}_{4}
\right) \varepsilon_{0} \right\} \| \tilde{u} -\tilde{v} \|_{X_{2}} \\
& \le \max \{ 3 C_{0} \tilde{C}_{2}, 3(C_{3} +C_{4})+\tilde{C}_{3} +\tilde{C}_{4} \} 
\left( \| f_{0}, f_{1} \|_{Y_{1}} + \varepsilon_{0} \right)
\| \tilde{u} -\tilde{v} \|_{X_{2}} \\
& \le \frac{1}{2} \| \tilde{u} -\tilde{v} \|_{X_{2}} 
\end{split}
\end{equation*}
for $0< \varepsilon_{1} \le \varepsilon_{1}'$ and $\| f_{0}, f_{1} \|_{Y_{1}} + \varepsilon_{0} < \varepsilon_{1}$, 
which shows the desired estimate \eqref{eq:6.3}.  
\subsection{Low frequency estimates}
In this subsection we collect the proof of the estimates of low frequency parts of $I_{3}[\tilde{u}](t)$ and $I_{4}[\tilde{u}](t)$.
We begin with the estimates for $I_{3L}[\tilde{u}](t)$.
To do so, we apply the estimates \eqref{eq:3.2} and \eqref{eq:2.13} to see that
\begin{equation} \label{eq:6.9}
\begin{split}
\| \nabla  I_{3L}[\tilde{u}](t) \|_{2} & \le C \int_{0}^{t} (1+t-s)^{-\frac{3}{4}} \| \nabla \tilde{u}(s) \nabla^{2} u_{per}(s) \|_{1} ds \\
& \le C \int_{0}^{t} (1+t-s)^{-\frac{3}{4}} 
\left\| \nabla \tilde{u}(s) \right\|_{10} \left\| \nabla^{2} u_{per}(s) \right\|_{\frac{10}{9}} ds \\ 
& \le C  \varepsilon_{0} \| \tilde{u} \|_{X_{2}} \int_{0}^{t} (1+t-s)^{-\frac{3}{4}}  (1+s)^{-\frac{7}{4}} ds \\ 
& \le C  \varepsilon_{0} (1+t)^{-\frac{3}{4}} \| \tilde{u} \|_{X_2},
\end{split}
\end{equation}
where we used the fact that 
\begin{equation} \label{eq:6.10}
\begin{split}
\left\| \nabla \tilde{u}(s) \right\|_{10} \le 
\left\| \nabla \tilde{u}(s) \right\|_{\infty}^{\frac{4}{5}} 
\left\| \nabla \tilde{u}(s) \right\|_{2}^{\frac{1}{5}} 
\le C \| \tilde{u} \|_{X_2} (1+s)^{-\frac{7}{4}}.
\end{split}
\end{equation}
By the same way, we also have
\begin{equation} \label{eq:6.11}
\begin{split}
\| \nabla^{3}  I_{3L}[\tilde{u}](t) \|_{2} & \le C \int_{0}^{t} (1+t-s)^{-\frac{7}{4}} \| \nabla \tilde{u}(s) \nabla^{2} u_{per}(s) \|_{1} ds \\
& \le C \int_{0}^{t} (1+t-s)^{-\frac{7}{4}} 
\left\| \nabla \tilde{u}(s) \right\|_{10} \left\| \nabla^{2} u_{per}(s) \right\|_{\frac{10}{9}} ds \\ 
& \le C  \varepsilon_{0} \| \tilde{u} \|_{X_{2}} \int_{0}^{t} (1+t-s)^{-\frac{7}{4}}  (1+s)^{-\frac{7}{4}} ds \\ 
& \le C  \varepsilon_{0}
 (1+t)^{-\frac{7}{4}} \| \tilde{u} \|_{X_2}.
\end{split}
\end{equation}
Moreover, the estimates \eqref{eq:3.2}, \eqref{eq:2.13} and \eqref{eq:6.10} yield 
\begin{equation} \label{eq:6.12}
\begin{split}
& \| \nabla^{3} I_{3L}[\tilde{u}](t) \|_{\frac{5}{2}} +\| \nabla I_{3L}[\tilde{u}](t) \|_{\infty}  \\
& \le C \int_{0}^{\frac{t}{2}} (1+t-s)^{-2} \| \nabla \tilde{u}(s) \nabla^{2} u_{per}(s) \|_{1} ds \\
& \quad + C \int_{\frac{t}{2}}^{t} (1+t-s)^{-\frac{5}{2}(\frac{3}{4}-\frac{2}{5})-\frac{1}{2}} \| \nabla \tilde{u}(s) \nabla^{2} u_{per}(s) \|_{\frac{4}{3}} ds \\
& \le C \int_{0}^{\frac{t}{2}} (1+t-s)^{-2} 
\left\| \nabla \tilde{u}(s) \right\|_{10} \left\| \nabla^{2} u_{per}(s) \right\|_{\frac{10}{9}} ds \\
& + C \int_{\frac{t}{2}}^{t} (1+t-s)^{-\frac{11}{8}} \| \nabla \tilde{u}(s) \|_{\infty} \| \nabla^{2} u_{per}(s) \|_{\frac{4}{3}}
ds \\
& \le C \varepsilon_{0} \| \tilde{u} \|_{X_{2}} \int_{0}^{\frac{t}{2}} (1+t-s)^{-2}  (1+s)^{-\frac{7}{4}} ds \\ 
& +  C \varepsilon_{0} \| \tilde{u} \|_{X_{2}} \int_{\frac{t}{2}}^{t} (1+t-s)^{-\frac{11}{8}}  (1+s)^{-2} ds \\  
& \le C  \varepsilon_{0} (1+t)^{-2} \| \tilde{u} \|_{X_2}.
\end{split}
\end{equation}
Next we prove the estimates for $I_{4L}[\tilde{u}](t)$.
It is easy to see that 
\begin{equation} \label{eq:6.13}
\begin{split}
\| \nabla I_{4L}[\tilde{u}](t) \|_{2} & \le C \int_{0}^{t} (1+t-s)^{-\frac{3}{4}} \| \nabla u_{per}(s) \nabla^{2} \tilde{u} (s) \|_{1} ds \\
& \le C \int_{0}^{t} (1+t-s)^{-\frac{3}{4}} \| \nabla u_{per}(s) \|_{2} \| \nabla^{2} \tilde{u} (s) \|_{2} ds \\
& \le C  \varepsilon_{0} \| \tilde{u} \|_{X_2} \int_{0}^{t} (1+t-s)^{-\frac{3}{4}}(1+s)^{-\frac{5}{4}}  ds \\
& \le C  \varepsilon_{0} (1+t)^{-\frac{3}{4}} \| \tilde{u} \|_{X_2}
\end{split}
\end{equation}
by \eqref{eq:3.2}, \eqref{eq:2.12} and the estimate 
\begin{equation} \label{eq:6.14}
\begin{split}
\left\| \nabla^{2} \tilde{u}(s) \right\|_{2} \le 
\left\| \nabla^{3} \tilde{u}(s) \right\|_{2}^{\frac{1}{2}} 
\left\| \nabla \tilde{u}(s) \right\|_{2}^{\frac{1}{2}} 
\le C \| \tilde{u} \|_{X_2} (1+s)^{-\frac{5}{4}}.
\end{split}
\end{equation}
For the estimates of $\| \nabla^{3}  I_{4L}[\tilde{u}](t) \|_{2}$, 
$\| \nabla^{3} I_{4L}[\tilde{u}](t) \|_{\frac{5}{2}}$ and $\| \nabla I_{4L}[\tilde{u}](t) \|_{\infty}$,
we observe that 
\begin{equation} \label{eq:6.15}
\begin{split}
\| f g\|_{\frac{9}{7}} \le \| f \|_{\frac{45}{26}} \| g \|_{5}, 
\end{split}
\end{equation}
and
\begin{equation} \label{eq:6.16}
\begin{split}
\| \nabla^{2} \tilde{u}(s) \|_{5} \le C \| \nabla \tilde{u}(s)  \|_{\infty}^{\frac{1}{2}} \| \nabla^{3} \tilde{u}(s)  \|_{\frac{5}{2}}^{\frac{1}{2}} 
\le C (1+s)^{-2} \| \tilde{u} \|_{X_2}
\end{split}
\end{equation}
by \eqref{eq:3.10}.
Then we see that 
\begin{equation} \label{eq:6.17}
\begin{split}
\| \nabla^{3}  I_{4L}[\tilde{u}](t) \|_{2} & \le C \int_{0}^{\frac{t}{2}} (1+t-s)^{-\frac{7}{4}} \| \nabla u_{per}(s) \nabla^{2} \tilde{u}(s) \|_{1} ds \\
& \quad + \int_{\frac{t}{2}}^{t} (1+t-s)^{-\frac{5}{2}(\frac{7}{9}-\frac{1}{2})-\frac{1}{2}} \| \nabla u_{per}(s) \nabla^{2} \tilde{u}(s) \|_{\frac{9}{7}} ds \\
& \le C \int_{0}^{\frac{t}{2}} (1+t-s)^{-\frac{7}{4}} \| \nabla u_{per}(s) \|_{2} \| \nabla^{2} \tilde{u} (s) \|_{2} ds \\
& + C \int_{\frac{t}{2}}^{t} (1+t-s)^{-\frac{43}{36}}
 \| \nabla u_{per} (s) \|_{\frac{45}{26}} \| \nabla^{2} \tilde{u} (s) \|_{5}
ds \\
& \le C \varepsilon_{0} \| \tilde{u} \|_{X_2}  \int_{0}^{\frac{t}{2}} (1+t-s)^{-\frac{7}{4}} (1+s)^{-\frac{5}{4}} ds \\
& + C \varepsilon_{0} \| \tilde{u} \|_{X_2}  \int_{\frac{t}{2}}^{t} (1+t-s)^{-\frac{43}{36}} (1+s)^{-2}
ds \\ 
& \le C  \varepsilon_{0}
(1+t)^{-\frac{7}{4}} \| \tilde{u} \|_{X_2}
\end{split}
\end{equation}
by \eqref{eq:3.2}, \eqref{eq:2.12} and \eqref{eq:6.14}-\eqref{eq:6.16}.
Again, 
applying the estimates \eqref{eq:3.2}, \eqref{eq:2.12} and \eqref{eq:6.14}-\eqref{eq:6.16},
we also obtain 
\begin{equation} \label{eq:6.18}
\begin{split}
& \| \nabla^{3} I_{4L}[\tilde{u}](t) \|_{\frac{5}{2}}  + \| \nabla I_{4L}[\tilde{u}](t) \|_{\infty} \\
& \le C \int_{0}^{\frac{t}{2}} (1+t-s)^{-2} \| \nabla u_{per}(s) \nabla^{2} \tilde{u}(s) \|_{1} ds \\
& + C \int_{\frac{t}{2}}^{t} (1+t-s)^{-\frac{5}{2}(\frac{7}{9}-\frac{2}{5})-\frac{1}{2} } \| \nabla u_{per}(s) \nabla^{2} \tilde{u}(s) \|_{\frac{9}{7}} ds \\
& \le C \int_{0}^{\frac{t}{2}} (1+t-s)^{-2} \| \nabla u_{per}(s) \|_{2} \| \nabla^{2} \tilde{u}(s) \|_{2} ds \\
& + C \int_{\frac{t}{2}}^{t} (1+t-s)^{-\frac{13}{9}}  \| \nabla u_{per}(s) \|_{\frac{45}{26}} \| \nabla^{2} \tilde{u} (s) \|_{5} ds \\
& \le C  \varepsilon_{0} \| \tilde{u} \|_{X_2} \int_{0}^{t} (1+t-s)^{-2}(1+s)^{-\frac{5}{4}}  ds \\ 
& + C   \varepsilon_{0} \| \tilde{u} \|_{X_2}  \int_{\frac{t}{2}}^{t} (1+t-s)^{-\frac{13}{9}} (1+s)^{-2}
ds \\ 
& \le C  \varepsilon_{0} (1+t)^{-2}  \| \tilde{u} \|_{X_2}.
\end{split}
\end{equation}
Next, we prove the estimates for $\nabla \partial_{t} I_{3}[\tilde{u}](t)$.
Noting that $K_{1}(0,x)=0$,
we apply the same argument for the estimate of $\| \nabla  I_{3}[\tilde{u}](t) \|_{2}$ to see that
\begin{equation} \label{eq:6.19}
\begin{split}
\| \nabla \partial_{t} I_{3L}[\tilde{u}](t) \|_{2} 
& \le C \varepsilon_{0} (1+t)^{-\frac{5}{4}} \| \tilde{u} \|_{X_2}, \\
\| \nabla \partial_{t} I_{3L}[\tilde{u}](t) \|_{\frac{5}{2}} 
& \le C \varepsilon_{0} (1+t)^{-\frac{3}{2}} \| \tilde{u} \|_{X_2}, \\
\| \partial_{t} I_{3L}[\tilde{u}](t) \|_{2} 
& \le C \varepsilon_{0} (1+t)^{-\frac{3}{4}} \| \tilde{u} \|_{X_2}, \\
\| \partial_{t} I_{3L}[\tilde{u}](t) \|_{\frac{5}{2}} 
& \le C \varepsilon_{0} (1+t)^{-1} \| \tilde{u} \|_{X_2}.
\end{split}
\end{equation}
By the same method in \eqref{eq:6.13}, 
we also have the estimates for $\| \nabla \partial_{t} I_{4L}[\tilde{u}](t) \|_{2}$, 
$\| \partial_{t} I_{4L}[\tilde{u}](t) \|_{2}$
and $\| \partial_{t} I_{4L}[\tilde{u}](t) \|_{\frac{5}{2}}$ 
as follows:
\begin{equation} \label{eq:6.20}
\begin{split}
\| \nabla \partial_{t} I_{4L}[\tilde{u}](t) \|_{2} & 
\le C \varepsilon_{0} (1+t)^{-\frac{5}{4}} \| \tilde{u} \|_{X_2}, \\
\| \partial_{t} I_{4L}[\tilde{u}](t) \|_{2} 
& \le C  \varepsilon_{0} (1+t)^{-\frac{3}{4}} \| \tilde{u} \|_{X_2}, \\
\| \partial_{t} I_{4L}[\tilde{u}](t) \|_{\frac{5}{2}} & 
\le C \varepsilon_{0} (1+t)^{-1} \| \tilde{u} \|_{X_2}.
\end{split}
\end{equation}
Finally we show the estimate for $\| \nabla \partial_{t} I_{4L}[\tilde{u}](t) \|_{\frac{5}{2}}$. 
Noting that 
$$
\| \nabla u_{per}(s) \nabla^{2} \tilde{u}(s) \|_{1} 
\le \| \nabla u_{per}(s) \|_{\frac{5}{3}} \| \nabla^{2} \tilde{u}(s) \|_{\frac{5}{2}}
$$
and 
\begin{equation*} 
\begin{split}
\| \nabla^{2} \tilde{u}(s) \|_{\frac{5}{2}} 
& \le C\| \nabla \tilde{u}(s) \|_{\frac{5}{2}}^{\frac{1}{2}}
\| \nabla^{3} \tilde{u}(s) \|_{\frac{5}{2}}^{^\frac{1}{2}} 
\le C \varepsilon_{1} (1+t)^{-\frac{3}{2}} \| \tilde{u} \|_{X_2},
\end{split}
\end{equation*}
we obtain 
\begin{equation} \label{eq:6.21}
\begin{split}
\| \nabla \partial_{t} I_{4L}[\tilde{u}](t) \|_{\frac{5}{2}}   
& \le C \int_{0}^{t} (1+t-s)^{-\frac{3}{2}} \| \nabla u_{per}(s) \nabla^{2} \tilde{u}(s) \|_{1} ds \\
& \le C \int_{0}^{t} (1+t-s)^{-\frac{3}{2}} 
\| \nabla u_{per}(s) \|_{\frac{5}{3}} \| \nabla^{2} \tilde{u}(s) \|_{\frac{5}{2}} ds \\
& \le C \varepsilon_{0} \| \tilde{u} \|_{X_2} 
\int_{0}^{t} (1+t-s)^{-\frac{3}{2}}(1+s)^{-\frac{3}{2}}  ds \\ 
& \le C  \varepsilon_{0} (1+t)^{-\frac{3}{2}}  \| \tilde{u} \|_{X_2} .
\end{split}
\end{equation}

We complete the estimates for the low frequency parts, $I_{3L}[\tilde{u}](t)$ and $I_{4L}[\tilde{u}](t)$. 
\subsection{Middle and high frequency estimates}
%
%
In this subsection, we prove the estimates for the middle and the high frequency parts, 
$I_{3M}[\tilde{u}](t)+I_{3H}[\tilde{u}](t)$ and $I_{4M}[\tilde{u}](t)+I_{4H}[\tilde{u}](t)$.
First, 
we estimate $\| \nabla^{3} ( I_{3M}[\tilde{u}](t) + I_{3H}[\tilde{u}](t) )\|_{2}$.
Noting that 
\begin{equation} \label{eq:6.22}
\begin{split}
\| \nabla (\nabla \tilde{u}(s) \nabla^{2} u_{per}(s)) \|_{2} & \le \| \nabla^{2} \tilde{u}(s) \nabla^{2} u_{per}(s) \|_{2} 
+\| \nabla \tilde{u}(s) \nabla^{3} u_{per}(s) \|_{2} \\
& \le \| \nabla^{2} \tilde{u}(s) \|_{5} \| \nabla^{2} u_{per}(s) \|_{\frac{10}{3}} + \| \nabla \tilde{u}(s)\|_{\infty} \| \nabla^{3} u_{per}(s) \|_{2} \\
& \le C \varepsilon_{0} \| \tilde{u} \|_{X_2} (1+s)^{-2}
\end{split}
\end{equation}
by the estimates \eqref{eq:2.13}, \eqref{eq:2.14}, \eqref{eq:6.16} and 
\begin{equation} \label{eq:6.23}
\| f g \|_{2} \le \| f \|_{5} \| g \|_{\frac{10}{3}},
\end{equation} 
we see that
\begin{equation} \label{eq:6.24}
\begin{split}
\| \nabla^{3} ( I_{3M}[\tilde{u}](t) + I_{3H}[\tilde{u}](t) )\|_{2} & 
\le C \int_{0}^{t} e^{-c(t-s)} \| \nabla (\nabla \tilde{u}(s) \nabla^{2} u_{per}(s)) \|_{2} ds \\
& \le C \varepsilon_{0} \| \tilde{u} \|_{X_2} \int_{0}^{t} e^{-c(t-s)} (1+s)^{-2} ds \\ 
& \le C  \varepsilon_{0}
 (1+t)^{-2} \| \tilde{u} \|_{X_2} 
\end{split}
\end{equation}
by \eqref{eq:3.4}, \eqref{eq:3.5} and \eqref{eq:6.22}. 
Similarly, observing that
\begin{equation} \label{eq:6.25}
\begin{split}
\| \nabla (\nabla u_{per}(s) \nabla^{2} \tilde{u}(s) ) \|_{2} & 
\le \| \nabla^{2} u_{per}(s) \nabla^{2} \tilde{u}(s) \|_{2} 
+\| \nabla u_{per}(s) \nabla^{3} \tilde{u}(s) \|_{2} \\
& \le \| \nabla^{2} u_{per}(s) \|_{\frac{10}{3}} \| \nabla^{2} \tilde{u}(s) \|_{5} + 
\| \nabla u_{per}(s) \|_{10} \| \nabla^{3} \tilde{u}(s)\|_{\frac{5}{2}} \\
& \le C \varepsilon_{1} \| \tilde{u} \|_{X_2} (1+s)^{-2}
\end{split}
\end{equation}
by the estimates \eqref{eq:6.23}, 
$
\| f g \|_{2} \le \| f \|_{\frac{5}{2}} \| g \|_{10}
$, 
\eqref{eq:6.16}, \eqref{eq:2.12} and \eqref{eq:2.13},
we apply the estimates \eqref{eq:3.4}, \eqref{eq:3.5} and \eqref{eq:6.25} to have
\begin{equation} \label{eq:6.26}
\begin{split}
\| \nabla^{3} ( I_{4M}[\tilde{u}](t) + I_{4H}[\tilde{u}](t) )\|_{2} & 
\le C \int_{0}^{t} e^{-c(t-s)} \| \nabla (\nabla u_{per}(s) \nabla^{2} \tilde{u}(s)) \|_{2} ds \\
& \le C \varepsilon_{0} \| \tilde{u} \|_{X_2} \int_{0}^{t} e^{-c(t-s)} (1+s)^{-2} ds \\ 
& \le C  \varepsilon_{0}
 (1+t)^{-2} \| \tilde{u} \|_{X_2}. 
\end{split}
\end{equation}
By the similar way,  
we also arrive at the estimates 
\begin{equation} \label{eq:6.27}
\begin{split}
\| \nabla ( I_{3M}[\tilde{u}](t) + I_{3H}[\tilde{u}](t) )\|_{2} 
+
\| \nabla ( I_{4M}[\tilde{u}](t) + I_{4H}[\tilde{u}](t) )\|_{2} 
\le C  \varepsilon_{0} (1+t)^{-2} \| \tilde{u} \|_{X_2}
\end{split}
\end{equation}
by \eqref{eq:3.4} taking $p=q=2$, 
$\alpha= \tilde{\alpha}=\alpha_{1}=1$ and $\ell=0$,
\begin{equation} \label{eq:6.28}
\begin{split}
\| \partial_{t} ( I_{3M}[\tilde{u}](t) + I_{3H}[\tilde{u}](t) )\|_{2} 
+
\| \partial_{t} ( I_{4M}[\tilde{u}](t) + I_{4H}[\tilde{u}](t) )\|_{2} 
& \le C  \varepsilon_{0}
 (1+t)^{-2} \| \tilde{u} \|_{X_2},
\end{split}
\end{equation}
by \eqref{eq:3.4} taking $p=q=2$, 
$\alpha=0$ and $\alpha_{1} = \ell = \tilde{\ell}=1$, and 
\begin{equation} \label{eq:6.288}
\begin{split}
& \| \nabla \partial_{t} ( I_{3M}[\tilde{u}](t) + I_{3H}[\tilde{u}](t) )\|_{2} 
+
\| \nabla \partial_{t} ( I_{4M}[\tilde{u}](t) + I_{4H}[\tilde{u}](t) )\|_{2} \\
& \le C  \varepsilon_{0}
 (1+t)^{-2} \| \tilde{u} \|_{X_2},
\end{split}
\end{equation}
by \eqref{eq:3.4} taking $p=q=2$, 
$\alpha=\alpha_{1}=\ell=1$ and $\tilde{\ell}=0$.
For the estimate of $\dot{W}^{3,\frac{5}{2}}$ norms, 
combining the estimates \eqref{eq:2.13}-\eqref{eq:2.14} and \eqref{eq:6.16} gives 
\begin{equation} \label{eq:6.29}
\begin{split}
\| \nabla (\nabla \tilde{u}(s) \nabla^{2} u_{per}(s)) \|_{\frac{5}{2}} & \le \| \nabla^{2} \tilde{u}(s) \nabla^{2} u_{per}(s) \|_{\frac{5}{2}} 
+\| \nabla \tilde{u}(s) \nabla^{3} u_{per}(s) \|_{\frac{5}{2}} \\
& \le \| \nabla^{2} \tilde{u}(s) \|_{5} \| \nabla^{2} u_{per}(s) \|_{5} + \| \nabla \tilde{u}(s)\|_{\infty} \| \nabla^{3} u_{per}(s) \|_{\frac{5}{2}} \\
& \le C \varepsilon_{0} (1+s)^{-2} \| \tilde{u} \|_{X_2}
\end{split}
\end{equation}
and 
\begin{equation} \label{eq:6.30}
\begin{split}
\| \nabla (\nabla u_{per}(s) \nabla^{2} \tilde{u}(s)) \|_{\frac{5}{2}} 
& \le \| \nabla^{2} u_{per}(s) \nabla^{2} \tilde{u}(s) \|_{\frac{5}{2}} 
+\| \nabla u_{per}(s) \nabla^{3} \tilde{u}(s) \|_{\frac{5}{2}} \\
& \le \| \nabla^{2} u_{per}(s) \|_{5} \| \nabla^{2} \tilde{u}(s) \|_{5} 
+ \| \nabla u_{per}(s) \|_{\infty} \| \nabla^{3} \tilde{u}(s)\|_{\frac{5}{2}} \\
& \le C \varepsilon_{0} (1+s)^{-2} \| \tilde{u} \|_{X_2},
\end{split}
\end{equation}
where we used the facts that $\| f g \|_{\frac{5}{2}} \le \| f \|_{5} \| g \|_{5}$.
Therefore it follows from \eqref{eq:3.4}, \eqref{eq:3.5}, \eqref{eq:6.29} and \eqref{eq:6.30} 
that 
\begin{equation} \label{eq:6.31}
\begin{split}
& \| \nabla^{3} ( I_{3M}[\tilde{u}](t) + I_{3H}[\tilde{u}](t) )\|_{\frac{5}{2}} +\| \nabla^{3} ( I_{4M}[\tilde{u}](t) + I_{4H}[\tilde{u}](t) )\|_{\frac{5}{2}} \\
& 
\le C \int_{0}^{t} e^{-c(t-s)} ( \| \nabla (\nabla \tilde{u}(s) \nabla^{2} u_{per}(s)) \|_{\frac{5}{2}} +\| \nabla (\nabla u_{per}(s) \nabla^{2} \tilde{u}(s) ) \|_{\frac{5}{2}}) ds \\
& \le C \varepsilon_{0} \| \tilde{u} \|_{X_2} \int_{0}^{t} e^{-c(t-s)} (1+s)^{-2} ds \\ 
& \le C  \varepsilon_{0}
 (1+t)^{-2} \| \tilde{u} \|_{X_2}. 
\end{split}
\end{equation}
By the same way, 
we also have 
\begin{equation} \label{eq:6.32}
\begin{split}
& \| \nabla \partial_{t} ( I_{3M}[\tilde{u}](t) + I_{3H}[\tilde{u}](t) )\|_{\frac{5}{2}} 
+\| \nabla \partial_{t} ( I_{4M}[\tilde{u}](t) + I_{4H}[\tilde{u}](t) )\|_{\frac{5}{2}} \\
& \le C  \varepsilon_{0}
 (1+t)^{-2} \| \tilde{u} \|_{X_2}
\end{split}
\end{equation}
by \eqref{eq:3.4} taking $p=q=\frac{5}{2}$, 
$\alpha=\alpha_{1}=\ell=1$ and $\tilde{\ell}=0$,
and
\begin{equation} \label{eq:6.33}
\begin{split}
\| \partial_{t} ( I_{3M}[\tilde{u}](t) + I_{3H}[\tilde{u}](t) )\|_{\frac{5}{2}} 
+
\| \partial_{t} ( I_{4M}[\tilde{u}](t) + I_{4H}[\tilde{u}](t) )\|_{\frac{5}{2}} 
& \le C  \varepsilon_{0}
 (1+t)^{-2} \| \tilde{u} \|_{X_2}
\end{split}
\end{equation}
by \eqref{eq:3.4} taking $p=q=\frac{5}{2}$, 
$\alpha=0$ and $\alpha_{1} = \ell = \tilde{\ell}=1$.
Finally we show the estimate of $\dot{W}^{1,\infty}$ norms.
Applying the estimate 
\begin{equation*}
\begin{split}
\| \nabla ( K_{1M}(t) + K_{1H}(t) ) \ast g \|_{\infty} \le Ce^{-ct} \| \nabla g \|_{2}
\end{split}
\end{equation*}
in \cite{K-T1}, \eqref{eq:6.22} and \eqref{eq:6.25},
we obtain
\begin{equation} \label{eq:6.34}
\begin{split}
& \| \nabla ( I_{3M}[\tilde{u}](t) + I_{3H}[\tilde{u}](t) )\|_{\infty} +\| \nabla ( I_{4M}[\tilde{u}](t) + I_{4H}[\tilde{u}](t) )\|_{\infty} \\
& 
\le C \int_{0}^{t} e^{-c(t-s)} ( \| \nabla (\nabla \tilde{u}(s) \nabla^{2} u_{per}(s)) \|_{2} +\| \nabla (\nabla u_{per}(s) \nabla^{2} \tilde{u}(s) ) \|_{2}) ds \\
& \le C \varepsilon_{0} \| \tilde{u} \|_{X_2} \int_{0}^{t} e^{-c(t-s)} (1+s)^{-2} ds \\ 
& \le C  \varepsilon_{0}
 (1+t)^{-2} \| \tilde{u} \|_{X_2}. 
\end{split}
\end{equation}
We complete the estimates for the middle and the high frequency parts of 
$I_{3}[\tilde{u}](t)$ and $I_{4}[\tilde{u}](t)$.
\subsection{Proof of stability of the time periodic solution}
As mentioned at the beginning of this section, 
we need to show the estimates \eqref{eq:6.7} and \eqref{eq:6.8}. 
We only prove the estimate \eqref{eq:6.7}.
The estimate \eqref{eq:6.8} is shown by the same way. 

By the estimates \eqref{eq:6.9} and \eqref{eq:6.27}, 
we see that 
\begin{equation} \label{eq:6.35}
\begin{split}
\| \nabla  I_{3}[\tilde{u}](t) \|_{2}  
& \le C  \varepsilon_{0} (1+t)^{-\frac{3}{4}} \| \tilde{u} \|_{X_2}.
\end{split}
\end{equation}
Similarly, we also have a series of the estimates:
\begin{equation} \label{eq:6.36}
\begin{split}
\| \nabla^{3}  I_{3}[\tilde{u}](t) \|_{2}  
& \le C  \varepsilon_{0} (1+t)^{-\frac{7}{4}} \| \tilde{u} \|_{X_2}
\end{split}
\end{equation}
by \eqref{eq:6.11} and \eqref{eq:6.26}, 
\begin{equation} \label{eq:6.37}
\begin{split}
\| \nabla^{3}  I_{3}[\tilde{u}](t) \|_{\frac{5}{2}} +\| \nabla I_{3}[\tilde{u}](t) \|_{\infty}   
& \le C  \varepsilon_{0} (1+t)^{-2} \| \tilde{u} \|_{X_2}
\end{split}
\end{equation}
by \eqref{eq:6.12}, \eqref{eq:6.31} and \eqref{eq:6.34} and 
\begin{equation} \label{eq:6.38}
\begin{split}
\| \nabla \partial_{t} I_{3}[\tilde{u}](t) \|_{2} 
& \le C  \varepsilon_{0} (1+t)^{-\frac{5}{4}} \| \tilde{u} \|_{X_2}, \\
\| \nabla \partial_{t} I_{3}[\tilde{u}](t) \|_{\frac{5}{2}} 
& \le C  \varepsilon_{0} (1+t)^{-\frac{3}{2}} \| \tilde{u} \|_{X_2}, \\
\| \partial_{t} I_{3}[\tilde{u}](t) \|_{2} 
& \le C  \varepsilon_{0} (1+t)^{-\frac{3}{4}} \| \tilde{u} \|_{X_2}, \\
\| \partial_{t} I_{3}[\tilde{u}](t) \|_{\frac{5}{2}} 
& \le C  \varepsilon_{0} (1+t)^{-1} \| \tilde{u} \|_{X_2}.
\end{split}
\end{equation}
by \eqref{eq:6.19}, \eqref{eq:6.28}, \eqref{eq:6.288}, \eqref{eq:6.32} and \eqref{eq:6.33}.
Thus we obtain the estimate \eqref{eq:6.7} for $j=3$ by \eqref{eq:6.35}-\eqref{eq:6.38}.
By the same manner,
we arrive at the estimate \eqref{eq:6.7} for $j=4$
by \eqref{eq:6.13}, \eqref{eq:6.17}, \eqref{eq:6.18}, \eqref{eq:6.20}, \eqref{eq:6.21},
\eqref{eq:6.26}-\eqref{eq:6.288} and \eqref{eq:6.31}-\eqref{eq:6.34}.    
We completes the proof of Theorem \ref{thm:2.2}.

\vspace*{5mm}
\noindent
\textbf{Acknowledgments. }
Y. Kagei was supported in part by JSPS KAKENHI Grant Numbers JP20H00118, JP24H00185.
H. Takeda was partially supported by JSPS KAKENHI Grant Numbers JP19K03596, JP24K06822.


\end{document}